\documentclass[11pt,reqno,sumlimits]{amsart} 
\usepackage{amssymb,amscd, epsfig, mathrsfs, xypic, amsmath,amsthm, color, enumerate}
\input xy \xyoption{all}
\CompileMatrices
\usepackage{eepic, xcolor}
\definecolor{grey}{rgb}{0.86, 0.86, 0.86}
\newcommand{\sfk}{{\mathsf{k}}}

\newcommand{\CH}{{C\!H}}

\newcommand{\bfs}{{\mathbf s}}

\newcommand{\bfx}{{\mathbf x }}

\newcommand{\bbI}{{\mathbb I}}

\newcommand{\LL}{{\mathbb L}}

\newcommand{\PP}{{\mathbb P}}

\newcommand{\bbS}{{\mathbb S}}

\newcommand{\ZZ}{{\mathbb Z}}

\newcommand{\frakG}{{\mathfrak G}}

\newcommand{\calB}{{\mathcal B}}

\newcommand{\calE}{{\mathcal E}}
\newcommand{\calF}{{\mathcal F}}

\newcommand{\calR}{{\mathcal R}}

\newcommand{\calU}{{\mathcal U}}

\newcommand{\bfF}{{\mathbf F}}

\newcommand{\lan}{{\langle}}
\newcommand{\ran}{{\rangle}}

\newcommand{\inc}{\hookrightarrow}

\newcommand{\Fl}{{Fl}}

\newcommand{\pt}{\operatorname{pt}}

\newcommand{\Spec}{\operatorname{Spec}}

\newcommand{\id}{{\operatorname{id}}}

\newcommand{\GL}{{\mbox{GL}}}

\newcommand{\uW}{{\underline{W}}}
\newcommand{\uw}{{\underline{w}}}
\newcommand{\uv}{{\underline{v}}}
\newcommand{\uu}{{\underline{u}}}
\newcommand{\uc}{{\underline{c}}}
\newcommand{\BS}{\mathrm{BS}}

\def \Xo {{\mathring{X}}}
\renewcommand{\a}{\mathbb{\alpha}}
\newcommand{\kk}{{{\mathsf{k}}}}
\newcommand{\vu}{{{\underline{v}}}}
\newcommand{\wu}{{{\underline{w}}}}
\newcommand{\Wu}{{{\underline{W}}}}
\DeclareMathOperator{\lo}{LO}
\DeclareMathOperator{\lp}{LP}
\DeclareMathOperator{\rp}{RP}
\newcommand{\scal}[1]{\langle #1 \rangle}

\newcommand{\Laz}{\mathbb{L}}
\newcommand{\SM}{\mathbf{Sm}_\sfk}

\newtheorem{thm}{Theorem}[section]  
\newtheorem{cor}[thm]{Corollary} 
\newtheorem{lem}[thm]{Lemma}  
\newtheorem{prop}[thm]{Proposition} 
\newtheorem{df-pr}[thm]{Definition-Proposition}

\theoremstyle{definition} 
\newtheorem{defn}[thm]{Definition}
   
\newtheorem{rem}[thm]{Remark}
 
\newtheorem{exm}[thm]{Example}

\numberwithin{equation}{section} 
\begin{document} 
\title{Stability of Bott--Samelson classes in algebraic cobordism}
\author{Thomas Hudson, Tomoo Matsumura and Nicolas Perrin}
\maketitle 
\begin{abstract}
In this paper, we construct {\it stable Bott--Samelson  classes} in the projective limit of the algebraic cobordism rings of full flag varieties, upon an initial choice of a reduced word in a given dimension. Each stable Bott--Samelson class is represented by a bounded formal power series modulo symmetric functions in positive degree. We make some explicit computations for those power series in the case of infinitesimal cohomology. We also obtain a formula of the restriction of Bott--Samelson classes to smaller flag varieties. 
\end{abstract}
\section{Introduction}
Let $\sfk$ be an algebraically closed field of characteristic $0$. Let $\Fl_n$ be the flag variety of complete flags in $\sfk^n$. It can be identified with the homogeneous space $\GL_n(\sfk)/B$ where $B$ is the Borel subgroup of upper triangular matrices. For each permutation $w\in S_n$, the corresponding Schubert variety $X_w^{(n)} \subset \Fl_n$ is defined as $\overline{B_-wB}$, the closure of the orbits of $wB$ by the action of  the opposite Borel subgroup $B_-$. If $\iota_n: \Fl_n \to \Fl_{n+1}$ is the natural embedding, the cohomology fundamental classes of these Schubert varieties have the property that $\iota_n^*[X_w^{(n+1)}] = [X_w^{(n)}]$, {\it i.e.}, {\it the Schubert classes are stable under the pullback maps}. The exact analogue of this property also holds in $K$-theory, in which one defines the Schubert classes as the $K$-theory classes of the structure sheaves of Schubert varieties.

In this paper, we attempt to generalize the above notion of stability to Bott-Samelson classes in algebraic cobordism. The algebraic cobordism, denoted  by $\Omega^*$, was introduced by Levine--Morel in \cite{LevineMorel} and represents the universal object among oriented cohomology theories, a family of functors which includes both the Chow ring $\CH^*$ and $K^0[\beta,\beta^{-1}]$, a graded version of the Grothendieck ring of vector bundles. 
  In recent years a lot of energy has been spent to lift results of Schubert calculus to $\Omega^*$, in the same way in which Bressler--Evens did in \cite{BresslerEvens1,BresslerEvens2} for topological cobordism. The first works in this direction were those of Calm\'es--Petrov--Zanoulline \cite{CalmesPetrovZanoulline} and Hornbostel--Kiritchenko \cite{HornbostelKiritchenkoCrelle} who investigated the algebraic cobordism of flag manifolds. Later, the interest shifted to Grassmann and flag bundles (cf. \cite{KiritchenkoKrishnaEquivariant},  \cite{CalmesZainoullineZhong},  \cite{HudsonMatsumura},\cite{HudsonMatsumura2},  \cite{HudsonMatsumurainf}, \cite{HornbostelPerrin}, \cite{Hudson1}, \cite{Hudson2}). One of the main difficulty of Schubert calculus in algebraic cobordism is caused by the fact that the fundamental classes of Schubert varieties are not well-defined in general oriented cohomology theories. A candidate for the replacement of Schubert classes is the family of the push-forward classes of Bott--Samelson resolutions of Schubert varieties. 

Since a Bott--Samelson variety is defined upon a choice of a reduced word,  thus our stability of Bott--Samelson classes depends on a particular choice of a {\it sequence of} reduced words. The followings are the main results in this paper: (1) For a given Bott--Samelson variety $Y_n$ in $\Fl_n$, we construct a sequence of Bott-Samelson varieties $Y_m$ over $\Fl_m$ for $m\geq n$  such that their push-forward classes in algebraic cobordism are stable under pullbacks, namely, the identity $\iota_m^* [Y_{m+1} \to \Fl_{m+1}] = [Y_m \to \Fl_m]$ holds in $\Omega^*(\Fl_m)$ for all $m\geq n$; (2) For a given Bott--Samelson variety $Y_n$ over $\Fl_n$, we find an explicit formula for the pullback $\iota_{n-1}^*[Y_n \to \Fl_n]$ of its push-forward class in $\Omega^*(\Fl_{n-1})$.

The pullback maps $\iota_n^*: \Omega^*(\Fl_{n+1}) \to \Omega^*(\Fl_{n})$ give rise to a projective system of graded rings. Based on the ring presentation of $\Omega^*(\Fl_{n})$ obtained by Hornbostel--Kiritchenko \cite{HornbostelKiritchenkoCrelle}, we observe that their graded projective limit, denoted by $\calR$, is isomorphic to the graded ring of {\it bounded formal power series} in an infinite sequence of variables $x=(x_i)_{i\in\ZZ_{>0}}$ with coefficients in the {\it Lazard ring} $\LL$ modulo the ideal of symmetric functions of positive degrees in $x$. Our stable sequence of Bott--Samelson classes determine a class in this limit, which we call a {\it stable Bott--Samelson class}. On each $\Omega^*(\Fl_n)$ the divided difference operators commute with the pullback maps and therefore lift to the limit $\calR$. This gives a method of computing the power series representing stable Bott--Samelson classes, which we apply to the case of a chosen infinitesimal cohomology theory. In particular, we obtain a formula for the power series representing stable Bott--Samelson classes associated to {\it dominant permutations}.

In \cite{HudsonMatsumura, HudsonMatsumura2}, the first and second authors obtained determinant formula of the cobordism push-forward classes of so-called Damon--Kempf--Laksov resolutions, generalizing the classical Damon--Kempf--Laksov determinant formula of Schubert classes. In \cite{HudsonMatsumurainf}, more explicit formula of Damon--Kempf--Laksov classes were obtained for infinitesimal cohomology. While these resolutions only exist for Schubert varieties associated to  vexillary permutations (like for instance Grassmannian elements), their push-forward classes are stable and so is their determinantal formula. On the other hand, Naruse--Nakagawa \cite{NakagawaNaruse1,NakagawaNaruse2, NakagawaNaruse3, NakagawaNaruse4} achieved, by considering a different resolution, a stable generalization of the Hall--Littlewood type formula for Schur polynomials in the context of topological cobordism. The differences among these stable expressions, including the ones obtained in this paper, should reflect the geometric nature of the different resolutions, each of which gives a different class in cobordism. 

The paper is organized as follows. 
In Section \ref{sec:prelim}, we recall basic facts about the algebraic cobordism ring of flag varieties and, in particular, we identify their projective limit.
In Section \ref{sec:stab}, we review the definition of Bott--Samelson resolutions and show the stability of their push-forward classes in cobordism based on the choice of a sequence of reduced words. We then focus on infinitesimal cohomology theory and compute, using divided difference operators, the power series representing the limits of the classes associated to dominant permutations.
In Section \ref{sec:rest}, we prove a formula for the product of any Bott--Samelson class with the class $[\Fl_{n-1} \to \Fl_n]$, generalizing the restriction formula given in Section \ref{sec:stab}.

\section{Preliminary}\label{sec:prelim}
Let $\sfk$ be an algebraic closed field of characteristic $0$. 
\subsection{Basics on algebraic cobordisms}\label{subsec:algcob}
For the reader's convenience, we will briefly recall some basic facts about algebraic cobordism and infinitesimal theories. More details on the construction and the properties of $\Omega^*$ can be found in \cite{LevineMorel}, while a more comprehensive description of $I_n^*$ is given in \cite{HudsonMatsumurainf}.

Both $\Omega^*$ and $I_n^*$ are examples of oriented cohomology theories, a family of contravariant functors $A^*:\SM \rightarrow \mathcal{R}^*$ from the category of smooth schemes to graded rings, which are furthermore endowed with push-forward maps for projective morphisms. Such functors are required to satisfy, together with some expected functorial compatibilities, the projective bundle formula and the extended homotopy property. These imply that, for every vector bundle $E\rightarrow X$, one is able to describe the evaluation of $A^*$ on the associated projective bundle $\PP(E)\rightarrow X$ as well as on every $E$-torsor $V\rightarrow X$. The Chow ring $\CH^*$ is probably the most well-known example of oriented cohomology theory and it should be kept in mind as a first approximation to the general concept.  

As a direct consequence of the projective bundle formula one has that every oriented cohomology theory admits a theory of Chern classes, which can be defined using Grothendieck's method. These satisfy most of the expected properties, like for instance the Whitney sum formula, however it is no longer true that the first Chern class behaves linearly with respect to tensor product: this is a key difference with $\CH^*$. For a pair of line bundles $L$ and $M$ defined over the same base, classically one has
\begin{eqnarray}\label{eqn additive law}
c^{\CH}_1(L\otimes M)=c^{\CH}_1(L)+c^{\CH}_1(M) \text{\ \ and\  \ }c_1^{\CH}(L^\vee)=-c_1^{\CH}(L),       
\end{eqnarray}
but these equalities in general fail for $c_1^A$. Instead, in order to describe $c_1^A(L\otimes M)$, it becomes necessary to introduce a \textit{formal group law}, a power series in two variables  defined over the coefficient ring $F_A\in A^*(\Spec\,k)[[u,v]]$ satisfying some requirements. Similarly, expressing $c_1^{A}(L^\vee)$ in terms of $c_1^{A}(L)$ requires one to consider the \textit{formal inverse} $\chi_A\in A^*(\Spec\, k)[[u]]$. The analogues of (\ref{eqn additive law}) then become  
\begin{eqnarray}
c^{A}_1(L\otimes M)=F_A(c^{A}_1(L),c^{A}_1(M)) \text{\ \ and\  \ }c_1^{A}(L^\vee)=\chi_A(c_1^{A}(L)).       
\end{eqnarray}
It is a classical result of Lazard \cite{Lazard} that every formal group law $(R,F_R)$ can be obtained from the universal one  $(\Laz,F_\Laz)$, which is defined over a ring later named after him. He also proved that, as a graded ring, $\Laz=\bigoplus_{m\leq 0} \LL^m$ is isomorphic to a polynomial ring in countably many variables $y_i$, each appearing in degree $-i$ for $i\geq 1$. In the case of a field of characteristic 0, Levine and Morel were able to prove that the coefficient ring of algebraic cobordism is isomorphic to $\Laz$ and that its formal group law $F_\Omega$ coincides with the universal one, which from now on we will simply denote $F$. The universality of $\Omega^*$ does not restrict itself only to its coefficient ring, in fact, Levine and Morel were able to prove the following theorem.

\begin{thm}[\protect{\cite[Theorem 1.2.6]{LevineMorel}}]
$\Omega^*$ is universal among oriented cohomology theories on $\SM$. That is, for any other oriented cohomology theory $A^*$ there exists a unique morphism 
$$\vartheta_A:\Omega^*\rightarrow A^*$$
of oriented cohomology theories.
\end{thm}
It essentially follows formally from this result that for any given formal group law $(R,F_R)$ the functor $\Omega^*\otimes_\Laz R$ is universal among the oriented cohomology theories with $R$ as coefficient ring and $F_R$ as associated law. This procedure can be used to produce functors, like the infinitesimal theories $I_n^*$, whose formal group laws are far simpler than the universal one and as a consequence more suitable for explicit computations. More precisely the projection $\Laz\rightarrow \ZZ[y_n]/(y_n^2)$, which maps $y_i$ to 0 unless $i=n$, gives rise to the following formal group law $F_{I_n}$ on $\ZZ[y_n]/(y_n^2)$: 
\begin{eqnarray}\label{eqn formal group inf}
F_{I_n}(u,v)=u+v+y_n\cdot \frac{1}{d_n}\sum_{j=1}^n \binom{n+1}{j}u^{j}v^{n+1-j}.
\end{eqnarray}
Here one has $d_n=p$, if $n+1$ is a power of a prime $p$, and $d_n=1$ otherwise. In our computations we will only consider the case $n=2$, for which (\ref{eqn formal group inf}) becomes 
$$u\boxplus v:=F_{I_2}(u,v)= u+v+y_2(u^2v+uv^2)=(u+v)(1+y_2 uv)$$
with the formal inverse being $\boxminus u:=\chi_{I_2}(u)=-u$. For the remainder of the paper we will write $\gamma$ instead of $y_2$.          

Let us finish this overview by discussing fundamental classes, another aspect in which a general oriented cohomology theory differs from $\CH^*$. While in $\CH^*$ it is possible to associate such a class to every equi-dimensional scheme, for a general oriented cohomology theory $A^*$ one has to restrict to schemes whose structure morphism is a local complete intersection. In particular, since not all Schubert varieties satisfy this requirement, it becomes necessary to find an alternative definition for Schubert classes. One possible option is to choose a family of resolutions of singularities and replace the fundamental classes of Schubert varieties with the pushforwards of the associated resolutions.           
\subsection{Algebraic cobordism of flag varieties and their limit}\label{subsec:flags}
For any integers $a,b$ such that $a\leq b$, let $[a,b]:=\{a,a+1,\dots, b\}$. Let $\sfk^{\ZZ_{>0}}$ be the infinite dimensional vector space generated by a formal basis $(e_i)_{i \in \ZZ_{>0}}$. For each $m\in \ZZ_{>0}$, let $E_m$ be the subspace of $\sfk^{\ZZ_{>0}}$ generated by $e_1,\dots, e_m$. We set $E_0=0$. We often identify $E_m$ with $\sfk^m$ the space of column vectors.

For each $n\in \ZZ_{>0}$, the flag variety $\Fl_n$ consists of flags $U_{\bullet}=(U_i)_{i \in [1,n-1]}$ of subspaces in $E_n$ where $U_i \subset U_{i+1}$ and $\dim U_i =i$ for each $i\in [1,n-1]$. Note that this implies $U_n=E_n$. For a fixed $n$, let $\calU_i^{(n)}, i\in [0,n]$ denote the tautological vector bundles of $\Fl_n$  and $\calE_i$ the trivial bundles of fiber $E_i$. In particular, $\calU_0^{(n)}=0$ and $\calU_n^{(n)}=\calE_n$. 

Let $\GL_n(\sfk)=\GL(E_n)$ be the general linear group. We consider the maximal torus $T_n \subset \GL_n(\sfk)$ given by the matrices having $(e_i)_{i \in [1,n]}$ as a basis of eigenvectors and the Borel subgroup $B_n \subset \GL_{n}(\sfk)$ given by the upper triangular matrices stabilizing the flag $E_\bullet = (E_i)_{i\in [1,n-1]}$ in $\Fl_n$. We can identify $\Fl_n$ with the homogeneous space $\GL_n(\sfk)/B_n$ by associating the matrix $M=(u_1,\dots, u_n)$ to a flag $U_{\bullet}$ where $\{u_j\}_{j\in [1,i]}$ is a basis of $U_i$. 

There is an isomorphism of graded rings (\cite[Thm 1.1]{HornbostelKiritchenkoCrelle})
\begin{equation}\label{isoiota}
\Omega^*(\Fl_n) \cong \LL[x_1,\dots, x_n]/\bbS_n  
\end{equation}
sending $c_1((\calU_i^{(n)}/\calU_{i-1}^{(n)})^{\vee})$ to $x_i$, where $\bbS_n$ is the ideal generated by the homogeneous symmetric polynomials in $x_1,\dots, x_n$ of strictly positive degree.

Let $\iota_n: \Fl_n\inc \Fl_{n+1}$ be the embedding induced by the canonical inclusion $E_n\inc E_{n+1}$. We have $\iota_n^* \calU_i^{(n+1)} = \calU_i^{(n)}$ for all $i\in [1,n]$ and $\iota_n^* \calU_{n+1}^{(n+1)} = \calE_{n+1}$. As a consequence, under the isomorphism (\ref{isoiota}), the pullback map $\iota_n^*: \Omega^*(\Fl_{n+1}) \to \Omega^*(\Fl_n)$ is the natural projection given by setting $x_{n+1}=0$. For each $m\in \ZZ$, let $\calR^m$ be the projective limit of $\Omega^m(\Fl_n)$ with respect to $\iota_n^*$. We define the {\it graded projective limit} of $\Omega^*(\Fl_n)$ with respect to $\iota_n^*$ to be $\calR:=\bigoplus_{m\in \ZZ} \calR^m$. 

In order to give a ring presentation of $\calR$, we introduce the following ring of formal power series. Let $x=(x_i)_{i\in \ZZ_{>0}}$ be a sequence of infinitely many indeterminates. Let $\ZZ^{\infty}$ be the set of infinite sequence $\bfs=(s_i)_{i\in \ZZ_{>0}}$ of nonnegative integers such that all but finitely many $s_i$'s are zero. Let $\LL[[x]]^{(m)}$ be the space of formal power series of degree $m\in \ZZ$. An element $f(x)$ of $\LL[[x]]^{(m)}$ is uniquely given as
\[
f(x) = \sum_{\bfs\in \ZZ^{\infty}} a_{\bfs}x^{\bfs}, \ \ \ a_{\bfs}\in \LL, \ \ x^{\bfs} = \prod_{i=1}^{\infty} x_i^{s_i}
\]
such that $|\bfs| + \deg a_{\bfs} = m$ where $|\bfs|=\sum_{i=0}^{\infty} s_i$ and $\deg a_{\bfs}$ is the degree of $a_{\bfs}$ in $\LL$. An element $f(x)\in \LL[[x]]^{(m)}$ is {\it  bounded} if $p_n(f(x))\in \LL[x_1,\dots,x_n]^{(m)}$, where $p_n$ is the substitution of $x_k=0$ for all $k>n$ and $\LL[x_1,\dots, x_n]^{(m)}$ is the degree $m$ part of $\LL[x_1,\dots,x_n]$.  Let $\LL[[x]]_{bd}^{(m)}$ be the set of all such bounded elements of $\LL[[x]]^m$. We set 
\[
\LL[[x]]_{bd}:=\bigoplus_{m\in \ZZ} \LL[[x]]_{bd}^{(m)}.
\]
This is a graded sub $\LL$-algebra of the ring $\LL[[x]]$ of formal power series.
\begin{prop}\label{RepPower}
There is an isomorphism of graded $\LL$-algebras
\[
\calR \cong \LL[[x]]_{bd}/\bbS_{\infty}
\]
where $\bbS_{\infty}$ is the ideal of $\LL[[x]]_{bd}$ generated by symmetric functions in $x$ of strictly positive degree.
\end{prop}
\begin{proof}
Let $m\in \ZZ$. The projections $p_n:\LL[[x]]_{bd}^{(m)} \to \LL[x_1,\dots, x_n]^{(m)}$ for $n>0$ induce a surjective homomorphism 
\[
\Phi:\LL[[x]]_{bd}^{(m)} \to \lim_{n\to \infty}\LL[x_1,\dots, x_n]^{(m)}
\]
sending $f(x)$ to $\{p_n(f(x))\}_{n\in \ZZ_{>0}}$. It is also easy to see that $\Phi$ is injective, and thus an isomorphism. Moreover, $p_n$'s induce surjections 
\[
\LL[[x]]_{bd}^{(m)} \cap \bbS_{\infty} \to \LL[x_1,\dots, x_n]^{(m)} \cap \bbS_n, \ \ \ \ n>0,
\]
inducing a bijection
\[
\LL[[x]]_{bd}^{(m)} \cap \bbS_{\infty} \cong \lim_{n\to \infty} \left(\LL[x_1,\dots, x_n]^{(m)} \cap \bbS_n\right).
\]
Thus we obtain the isomorphism
\[
\bigoplus_{m\in \ZZ} \LL[[x]]_{bd}^{(m)}/(\LL[[x]]_{bd}^{(m)} \cap \bbS_{\infty}) \cong \bigoplus_{m\in \ZZ} \lim_{n\to \infty} \LL[x_1,\dots, x_n]^{(m)}/(\LL[x_1,\dots, x_n]^{(m)} \cap \bbS_n),
\]
which is the desired one.
\end{proof}
\begin{defn}
An element in  $\calR^i$  is a sequence $\left(\alpha_n\right)_{n\in \ZZ_{>0}}$ such that $\alpha_n \in \Omega^i(\Fl_n)$ and $\iota_n^*(\alpha_{n+1})=\alpha_n$ for all $n>0$. An element of $\calR$ is  a finite linear combinations of such sequences and we call it a {\it stable class}.
\end{defn}
\begin{rem}
In order to specify an element of $\calR^i$, we only need to provide $\alpha_i$ for all $i \geq N$ for some fixed integer $N$. In fact, for $i<N$ the elements $\alpha_i$ can be obtained from $\alpha_N$ by applying the projections $\iota_n^*$.
\end{rem}
\subsection{Divided difference operators}
Let $W_n$ be the Weyl group of $\GL_{n}(\sfk)$. The maximal torus $T_{n}$ and the Borel subgroup $B_{n}$ define a system of simple reflections $s_1,\cdots,s_{n-1} \in W_n$ and we can identify $W_n$ with the symmetric group $S_{n}$ in $n$ letters, where each $s_i$ corresponds to the transposition of the letters $i$ and $i+1$. We denote the length of $w$ by $\ell(w)$.

For each $i\in [1,n-1]$, the {\it divided difference operator} $\partial_i$ is an operator on $\Omega^*(\Fl_n)$ defined as follows. Let $\Fl_n^{(i)}$ be the partial flag variety consisting of flags of the form $U_1\subset \cdots \subset U_{i-1} \subset U_{i+1}\subset\cdots \subset U_{n-1}$ with $\dim U_k =k$. Denote the canonical projection $\Fl_n \to \Fl_n^{(i)}$ by $p_i$. Then define $\partial_i := p_{i*}\circ p_i^*$. It is known from \cite{HornbostelKiritchenkoCrelle} that under the presentation (\ref{isoiota}), we have
\begin{equation}\label{ddo}
\partial_i (f(x)) = (\id + s_i)\frac{f(x)}{F(x_i,\chi(x_{i+1}))} = \frac{f(x)}{F(x_i,\chi(x_{i+1}))} + \frac{s_if(x)}{F(x_{i+1},\chi(x_{i}))}.
\end{equation}
\begin{lem}
The pullback $\iota_n^*: \Omega^*(\Fl_{n+1}) \to \Omega^*(\Fl_{n})$ commutes with $\partial_i$ for all $i\in [1,n-1]$. In particular, this shows that $\partial_i$ can be defined in the projective limit $\calR$ and it is given by the formula (\ref{ddo}). 
\end{lem} 
\begin{proof}
For each $i\in [1,n-1]$, $\iota_n$ and $p_i$ form a fiber diagram
\[
\xymatrix{
\Fl_n \ar[r]_{\iota_n} \ar[d]_{p_i}& \Fl_{n+1}\ar[d]_{p_i}\\
\Fl_n^{(i)} \ar[r]_{\iota_n} & \Fl_{n+1}^{(i)},
}
\]
and, since they are transverse, we have $\iota_n^*\circ p_{i*} = p_{i*}\circ\iota_n^*$. Thus $\iota_n^*\circ\partial_i=\iota_n^* \circ p_{i*}\circ p_i^* = p_{i*}\circ \iota_n^* \circ p_i^* = p_{i*} \circ p_i^*\circ \iota_n^*= \partial_i \circ \iota_n^*$.
\end{proof}

For a permutation $w \in W_{n}$, let $\Xo_w^{(n)} = B_n \cdot w(E_\bullet)$  be the Bruhat cell associated to $w$ in $\Fl_n$, where $w(E_{\bullet})$ is the flag consisting of $w(E_i)=\lan e_{w(1)},\dots, e_{w(i)}\ran$ for each $i\in [1,n-1]$. The {\it Schubert varieties} $X_w^{(n)}$ are the closures of the Bruhat cells: $X_w^{(n)} := \overline{B_n \cdot w(E_\bullet)}$. The {\it opposite Schubert varieties} are defined via $X^w_{(n)} = w_0 \cdot X_{w_0w}^{(n)}$, where $w_0=w_0^{(n)}$ is the longest element of $W_n$. As an orbit closure, we have  $X^w_{(n)} = \overline{B_n^-\cdot w(E_\bullet)}$ where $B_n^-:=w_0B_nw_0$ is the opposite Borel subgroup of lower triangular matrices. 

\begin{rem}
The fundamental class $[X^w_{(n)}]$ of $X^w_{(n)}$ is well-defined in the Chow ring of $\Fl_n$. Those classes are stable along pullbacks, {\it i.e.},  $\iota_n^*[X^w_{(n+1)}] = [X^w_{(n)}]$ in $\CH^*(\Fl_n)$ where $w \in S_n$ is regarded as an element of $S_{n+1}$ under the natural embedding $S_n\subset S_{n+1}$. As it is well-known,  its stable limit can be identified with the Schubert polynomial of Lascoux--Sch{\"u}tzenberger \cite{LascouxSchutzenbergerA}. It is also worth mentioning that the Schubert classes admit the following compatibility with divided difference operators, reflected on the definition of Schubert polynomials: for each $i\in[1,n-1]$, we have
\[
\partial_i [X^w_{(n)}] = \begin{cases}
[X^{ws_i}_{(n)}] & \ell(ws_i)=\ell(w) + 1, \\
0 & \mbox{otherwise}.
\end{cases}
\]
\end{rem}
\subsection{Some facts on permutations and reduced words}
We conclude this section by fixing notations for reduced words and showing a few lemmas and a proposition that will be used in the rest of the paper.

We denote by $\uW_n$ the set of words in $s_1,\dots, s_{n-1}$:  an element of $\uW_{n}$ will be written as a finite sequence $s_{i_1} \cdots s_{i_r}$, while the empty word is denoted by $1$. The {\it length} of a word $\uw=s_{i_1} \cdots s_{i_r}$ is the number $r$ of the letters $s_i$'s in $\uw$ and we denote it by $\ell(\uw)$.  For a word $\uw\in \uW_n$, we denote the corresponding permutation by $w \in W_n$. Let  $W_n^i$ be the subgroup of $W_{n}$ generated by all simple reflections $s_j$ with $j\not=i$ and $\uW_n^i$ the corresponding set of words. In particular, we can identify $W_n$ with $W_{n+1}^n$ and $\uW_n$ with $\uW_{n+1}^n$.

We denote the Bruhat order in $W_n$ by $\leq$, {\it i.e.}, $w\leq v$ if and only if every reduced word for $v$ contains a subword which is a reduced word for $w$. 
  

We denote by $c^{(n)}$ the {\it Coxeter element} $s_1\cdots s_n$ of $W_{n+1}$. It has a unique reduced word $\uc^{(n)}=s_1\cdots s_n$. Note that $\uc^{(n)}\uc^{(n-1)}\cdots \uc^{(1)}$ is a reduced word for the longest element $w_0^{(n+1)}$ of $W_{n+1}$.
\begin{lem} \label{lem-cn1}
If $\uc^{(n)}\uv \in \uW_{n+1}$ is a reduced word, then $\uv$ is a reduced word in  $\uW_n$.
\end{lem}
\begin{proof}
There exists a reduced word $\uu$ such that $\uc^{(n)}\uv\, \uu = \uw_0^{(n+1)}$ is a reduced word for the longest element $w_0^{(n+1)} \in W_{n+1}$. Since  $vu = (c^{(n)})^{-1}w_0^{(n+1)}=w_0^{(n)}$, we have $vu \in W_n$ so that any reduced word of $vu$ lies in $\uW_n$ and in particular $\uv\,\uu$ is a reduced word in $\uW_n$. Thus $\uv$ is a reduced word in $\uW_n$.
\end{proof}
\begin{lem}\label{lem-cn2}
If $\uv\in \uW_n$ is a reduced word, then $\uc^{(n)}\uv \in \uW_{n+1}$ is a reduced word. In particular, if $v=w_0^{(n)}w$ for some $w\in W_n$, then $\uc^{(n)}\uv$ is a reduced word for $w_0^{(n+1)}w$.
\end{lem}
\begin{proof}
There exists a reduced word $\uu$ such that $\uv \,\uu$ is a reduced word for $w_0^{(n)}$. Then $\uc^{(n)}\uv\,\uu$ is a reduced word for $w_0^{(n+1)}$. This implies that $\uc^{(n)}\uv$ is a reduced word.
\end{proof}
\begin{prop}\label{prop:w=ucv}
Let $w \in W_{n+1}$ such that $c:=c^{(n)}\leq w$. Every reduced word $\uw \in \uW_{n+1}$ for $w$ decomposes, modulo commuting relations, as $\uw = \uu\, \uc\, \uv$ with $\uu \in \uW_{n+1}^1$ and $\uv \in \uW_n$. 
\end{prop}
\begin{proof}
In this proof, all the equalities of words are modulo commuting relations. By definition of the Bruhat order, $\uw$ contains as a subword $\uc$, the unique reduced word of $c$. We choose such a subword by selecting the first occurrence of $s_1$, the first occurrence of $s_2$ after the chosen $s_1$ and so on. We thus have a decomposition
\[
\uw = \uw_1 s_1 \uw_2 s_2 \uw_3 \cdots \uw_n s_n \uw_{n+1}
\]
with $\uw_i \in \uW_{n+1}^i$ for $i \in [1,n]$. We have $\uw_i = \uv_i \uu_i$, where $\uv_i$ is a word in the $s_k$'s for $1\leq k<i$ and $\uu_i$ is a word in the $s_k$'s for $i<k\leq n$. Observing that  $\uv_i \uu_j = \uu_j \uv_i$ and $s_{i-1} \uu_j = \uu_j s_{i-1}$ for $i \leq j$, we thus obtain 
\[
\uw 
= \uw_1(\uu_1\uu_2\cdots \uu_n)(s_1 \uv_2 s_2 \uv_3 \cdots \uv_n s_n)\uw_{n+1}.
\]
For each $i \in [2,n]$, we claim that the word $\uv_i$ does not contain $s_{i-1}$, {\it i.e.}, it is a word in the $s_k$'s for $1\leq k\leq i-2$. We prove the claim by induction on $i$. First of all, it is easy to see that $\uv_2$ is an empty word since it is a word in $s_1$ only, and there is $s_1$ on the left of $\vu_2$ in the word $\uw$. Now by assuming that the claim holds for $i \leq k$, we have
\[
\uw = \uw_1(\uu_1\uu_2\cdots \uu_n) (s_1s_2 \cdots s_k) (\uv_1\cdots \uv_k)(\uv_{k+1} s_{k+1} \cdots \uv_n s_n) \uw_{n+1}.
\]
Since $s_1s_2 \cdots s_k \uv_1\cdots \uv_k\uv_{k+1} \in \uW_{k+1}$ is reduced, Lemma \ref{lem-cn1} implies that $\uv_1 \cdots \uv_{k+1} \in \uW_k$ and, in particular, we find that $\uv_{k+1}$ doesn't contain $s_k$. Thus the claim holds and by moving all $\uv_i$ to the right using commuting relations, we obtain
\[
\uw = \uw_1\uu_1\cdots \uu_n s_1s_2 \cdots s_n \uv_1\cdots \uv_n \uw_{n+1}.
\]
Using Lemma \ref{lem-cn1} again, we obtain $\uv_1\cdots \uv_n \uw_{n+1} \in \uW_n$, proving the proposition.
\end{proof}

\section{Stable Bott--Samelson classes}\label{sec:stab}
In this section, we introduce stable Bott--Samelson classes in the limit $\calR$ of $\Omega^*(\Fl_n)$. 
We also compute some of those classes explicitly in the case of infinitesimal cohomology.
\subsection{The stability of Bott--Samelson classes}


A Schubert variety is, in general, normal and Cohen-Macaulay, and has rational singularities. There exists several resolutions of singularities for it. We will be interested in the so-called Bott--Samelson resolutions.

We set $F_i^{(n)}:=\lan e_n,\dots, e_{n+1-i}\ran$ and denote the trivial bundle with fiber $F_i^{(n)}$ by $\calF_i^{(n)}$. 
\begin{defn}
For a reduced word $\uv=s_{i_1}\cdots s_{i_r} \in \uW_{n}$,  the {\it Bott--Samelson variety} $Y_{\underline{v}}^{(n)}$ is a subvariety of $(\Fl_n)^r$ defined as follows: 
\[
Y_{\underline{v}}^{(n)} = \left\{ (U_{\bullet}^{[0]}, U_{\bullet}^{[1]}, \dots, U_{\bullet}^{[r]}) \in (\Fl_n)^r\ \Big| \ 
\begin{array}{c}
U_i^{[k-1]} = U_i^{[k]}, \forall k = [1, r], \forall i\in[1,n-1] \backslash \{i_k\}
\end{array}
\right\},
\]
where $U_{\bullet}^{[0]} = F_{\bullet}^{(n)}$. If there is no confusion, we will sometimes write $Y_\vu$ for $Y_\vu^{(n)}$.
\end{defn}

\begin{rem}
In Definition \ref{def:BS2} we will give another equivalent construction (denoted $X_\wu$) of the Bott--Samelson resolutions.
\end{rem}

It is well-known (cf.\,\cite{Demazure1974}) that $Y_{\uv}$ is a smooth projective variety of dimension $r$. Let $\pi_n: (\Fl_n)^r \to \Fl_n$ be the projection to the $r$-th component. If $w\in W_n$ and $v=w_0^{(n)}w$, the projection $\pi_n$ induces a birational map $Y_{\uv} \to X^w$, which we refer to as a {\it Bott--Samelson resolution} of $X^w \subset \Fl_n$. 
\begin{thm} \label{thm:restriction}
Let $\uv \in \uW_n$ be a reduced word. There is a fiber diagram
\[
\xymatrix{
Y_{\uv}^{(n)}\ar[rr]_{\tilde\iota_n}\ar[d]_{\pi_{n}}&&Y_{\uc^{(n)}\uv}^{(n+1)}\ar[d]^{\pi_{n+1}}\\
\Fl_{n}\ar[rr]_{\iota_n}&&\Fl_{n+1},
}
\]
and we have $\iota_n^*\left(\left[Y_{\uc^{(n)}\uv} \to \Fl_{n+1}\right] \right)= \left[Y_{\uv} \to \Fl_n\right]$. 

Furthermore, let $\uc^{[n+m]}:=\uc^{(n+m-1)}\cdots \uc^{(n+1)}\uc^{(n)}$ where $\uc^{[n]}=1$, then the sequence
\[
\left[Y_{\uc^{[n+m]}\uv} \to \Fl_{n+m}\right], \ \ \ m\geq 0
\]
defines a stable class in $\calR$, which we call a {\it stable Bott--Samelson class} associated to $w$ and denote by $\BS_w^{\uv}$ if $v=w_0^{(n)}w$. 
\end{thm}
\begin{proof}
First we note that, by definition, an element of $Y_{\underline{v}}^{(n)}$ can be specified by a sequence of subspaces $(V_1,\dots, V_r)$ where $V_k=U^{[k]}_{i_k}$. We show that the map $\tilde\iota_n: Y_{\uv}^{(n)}\to Y_{\uc^{(n)}\uv}^{(n+1)}$ defined by 
\[
\tilde\iota_n(V_1,V_2,\dots, V_r):=(F_1^{(n)},\dots, F_n^{(n)},V_1,\dots, V_r)
\]
gives the desired fiber diagram. If we write an element of $Y_{\uc^{(n)}\uv}^{(n+1)}$ as 
\[
(A_{\bullet}^{[1]},\dots, A_{\bullet}^{[n]}, B_{\bullet}^{[1]},\dots, B_{\bullet}^{[r]}),
\] 
it suffices to show that $A_{k}^{[k]}=F_k^{(n)}$ for all $k\in [1,n]$ over the image of $\Fl_n$. Suppose that $B_{\bullet}^{[r]}$ is in the image of $\Fl_n$, then $B_n^{[r]}=E_n$. Since $i_1,\dots, i_r\in [1,n-1]$, we have $A_n^{[n]}=E_n=F_n^{(n)}$. We use backward induction on $k$ with the base case being $k=n$. Assume $A_{k+1}^{[k+1]} = F_{k+1}^{(n)}$. We  then have
\[
A_{k}^{[k]} \subset F_{k+1}^{(n+1)} \cap A_{k+1}^{[k+1]}=F_{k+1}^{(n+1)} \cap F_{k+1}^{(n)} = F_k^{(n)}.
\]
For the latter claim, we use the identity $\iota_n^*\pi_{n+1*}=\pi_{n*}\tilde\iota_n^*$ (see \cite[p.144 (BM2)]{LevineMorel}). We get
\begin{eqnarray*}
\iota_n^*\left[Y_{\uc^{(n)}\uv} \to \Fl_{n+1}\right]
=\iota_n^*\pi_{n+1*}(1_{Y_{\uc^{(n)}\uv}})
=\pi_{n*}\tilde\iota_n^*(1_{Y_{\uc^{(n)}\uv}})
=\pi_{n*}(1_{Y_{\uv}})
=[Y_{\uv} \to \Fl_n].
\end{eqnarray*}
This completes the proof of the claim.
\end{proof}
\begin{rem} 
We sometimes denote $\BS_w^{\uv}$ by $\BS_w^{\uv}(x)$ in order to stress that we regard it as an element of $\LL[x]_{bd}/\bbS_{\infty}$ under the identification in Proposition \ref{RepPower}.
\end{rem}
The following compatibility of Bott--Samelson classes with divided difference operators was established in \cite{HornbostelKiritchenkoCrelle}. 
\begin{lem}\label{lem3-3}
For a reduced word $\uv=s_{i_1}\cdots s_{i_r} \in \uW_{n}$, and $k\in [1,n-1]$, we have 
\[
\partial_i \left[Y_{\uv} \to \Fl_n\right] = \begin{cases}
\left[Y_{\uv s_i} \to \Fl_n\right]& \mbox{ if $\uv s_i$ is a reduced word}\\
0 & \mbox{otherwise}
\end{cases}
\]
\end{lem}
Since, as explained in Section \ref{subsec:flags}, the divided difference operators commute with the pullbacks $\iota_n^*$, we obtain the next corollary.
\begin{cor}\label{cor3-7}
Let $w\in W_n$ and set $v=w_0^{(n)}w$. Let $\uv$ be a reduced word of $v$. For any $i\in \ZZ_{>0}$, we have
\[
\partial_i \BS_w^{\uv} = \begin{cases}
\BS_{ws_i}^{\uv s_i} & \mbox{ if $\ell(ws_i)<\ell(w)$}\\
0 & \mbox{otherwise}.
\end{cases}
\]
\end{cor}

\begin{rem}
In the connective $K$-theory of $\Fl_n$, the class $[Y_{\uv} \to \Fl_n]$ coincides with the class of the opposite Schubert variety $X^w$, provided that $v=w_0^{(n)}w$. Its associated class in the projective limit is represented by the Grothendieck polynomial $\frakG_w(x)$ associated to $w$.
\end{rem}
\subsection{Examples in infinitesimal cohomology}
Throughout this section we will consider infinitesimal cohomology instead of algebraic cobordism. In combination with Proposition \ref{RepPower}, the use of this simpler theory will allow us to perform an explicit computation of the stable Bott--Samelson classes in terms of power series in $x$.

As in Section \ref{subsec:algcob}, the formal group law and its formal inverse for the infinitesimal cohomology $I_2^*$ are given by
\[
F_{I_2}(x,y) =x\boxplus y =   (x+ y) (1+ \gamma xy), \ \ \ \ \ \chi_{I_2}(x)=\boxminus x = -x
\]
with $\gamma^2=0$. We denote $I_2^*(\pt)= \ZZ[\gamma]/(\gamma^2)$ by $\bbI$. As explained in Section \ref{subsec:algcob}, we have
\[
I_2^*(\Fl_n) \cong \bbI[x_1,\dots, x_n]/\bbS_n
\]
where $\bbS_n$ is the ideal generated by the homogeneous symmetric polynomials of strictly positive degree in $x_1,\dots, x_n$. We set 
\[
\calR_{\bbI} :=\calR\otimes_{\LL}\bbI = \bbI[x]_{bd}/\bbS_{\infty}.
\]
By specialising (\ref{ddo}) to this particular case we obtain that on $\calR_{\bbI}$ the divided difference operator $\partial_i$ is given by
\[
\partial_i f = \frac{f - s_i f}{x_i-x_{i+1}} \cdot (1+\gamma x_ix_{i+1}), \ \ \ \ \ f\in \calR_{\bbI}.
\]

\begin{rem}
\begin{itemize}
\item[(1)] If $f$ is symmetric in $x_i$ and $x_{i+1}$, then $\partial_i (fg) = 
 f\partial_i g$ for all $g\in \calR_{\bbI}$.
\item[(2)] If $|i-j|\geq 2$, then $\partial_i\partial_j = \partial_j\partial_i$.
\end{itemize}
\end{rem}
For a reduced word $\uv=s_{i_1}\cdots s_{i_r}$, let $\partial_{\uv}:=\partial_{i_r}\cdots \partial_{i_1}$. Recall that $\uc^{(n)}=s_1\cdots s_n$.
\begin{lem}\label{lem5-1}
For $n\geq 1$, we have
\[
\partial_{\uc^{(n)}} \left( x_1^nx_2^{n-1}\cdots x_n\right) = \left( x_1^{n-1}x_2^{n-2}\cdots x_{n-1}\right)(1+\gamma e_2(x_1,\dots, x_{n+1})).
\]
\end{lem}
\begin{proof}
First we observe that
\begin{eqnarray}\label{eq5-1}
\partial_k(x_k(1+\gamma e_2(x_1,\dots,x_k)))
&=&1+\gamma e_2(x_1,\dots,x_{k+1})
\end{eqnarray}
which can be shown by a straightforward computation using the identities
\begin{eqnarray*}
e_2(x_1,\dots, x_k) &=& e_2(x_1,\dots, x_{k-1}) + x_k e_1(x_1,\dots, x_{k-1})\\
e_2(x_1,\dots, x_{k+1}) &=&e_2(x_k,x_{k+1}) +  e_2(x_1,\dots, x_{k-1}) +  e_1(x_1,\dots, x_{k-1})e_1(x_k,x_{k+1}).
\end{eqnarray*}
Now we prove the formula by induction on $n$. The case $n=1$ is obvious. 
If $n>1$, by induction hypothesis, we have
\begin{eqnarray*}
\partial_n\cdots\partial_1\left(  x_1^nx_2^{n-1}\cdots x_n\right)
&=&(x_1^{n-1}x_2^{n-2}\cdots x_{n-1})\partial_n(x_n(1+e_2(x_1,\dots,x_n))).
\end{eqnarray*}
Thus the claim follows from Equation \ref{eq5-1}.
\end{proof}
\begin{lem}\label{lem5-2}
Modulo $\bbS_N$, we have
\[
\sum_{k=n+1}^{N-1}  e_2(x_1,\dots, x_k) = - \sum_{i=n+1}^{N-1} (i-n)x_{i+1}e_1(x_1,\dots, x_i).
\]
\end{lem}
\begin{proof}
Let us begin by recalling the following identity of elementary symmetric polynomials
\[
e_2(x_1,\dots, x_N) = e_2(x_1,\dots, x_k) + \sum_{i=k}^{N-1} x_{i+1} e_1(x_1,\dots, x_i).
\]
Thus modulo $\bbS_N$, it follows that
\[
\sum_{k=n+1}^{N-1}  e_2(x_1,\dots, x_k) = -\sum_{k=n+1}^{N-1}\sum_{i=k}^{N-1} x_{i+1} e_1(x_1,\dots, x_i)
=-\sum_{i=n+1}^{N-1}\sum_{k=n+1}^{i}x_{i+1} e_1(x_1,\dots, x_i).
\]
The right hand side is the desired formula.
%
\end{proof}
For $w_0^{(n)}\in S_n$ the corresponding Schubert variety $X^{w_0^{(n)}}_{(n)}$ in $\Fl_n$ is a point, and so is the unique Bott--Samelson variety $Y_{1}^{(n)}$.  In $I_2^*(\Fl_n)$ we have
\[
\left[Y_1^{(n)} \to \Fl_n\right] =  x_1^{n-1}x_2^{n-2}\cdots x_{n-1}.
\]  
 The stable Bott--Samelson class $\BS_{w_0^{(n)}}^1$ introduced in Theorem \ref{thm:restriction} is given by the sequence
\[
\left[Y_{\uc^{(N-1)}\cdots \uc^{(n)}}^{(N)} \to \Fl_N\right] \in I_2^*(\Fl_N), \ \ \ N \geq n.
\]
By Lemma \ref{lem5-1} and \ref{lem5-2}, we can identify a formal power series representing this class in the ring $\calR_{\bbI}$ as follows.
\begin{thm}\label{thmH1}
In $\calR_{\bbI}$, we have
\begin{equation}\label{eq5-2}
\BS_{w_0^{(n)}}^{1}(x)=\left(  \prod_{i=1}^{n-1} x_i^{n-i}\right)  \left(1-\gamma \sum_{i=n+1}^{\infty} (i-n)x_{i+1}e_1(x_1,\dots, x_i)\right).
\end{equation}
\end{thm}
\begin{proof}
In view of Lemma \ref{lem3-3},  we can compute the class $\left[Y_{\uc^{(N-1)}\cdots \uc^{(n)}}^{(N)} \to \Fl_N\right]$ via divided difference operators:
\begin{eqnarray*}
\left[Y_{\uc^{(N-1)}\cdots \uc^{(n)}}^{(N)} \to \Fl_N\right]&=&\partial_{\uc^{(n)}}\cdots \partial_{\uc^{(N-1)}}\left[Y_{1}^{(N)} \to \Fl_N\right]\\
&=&\partial_{\uc^{(n)}}\cdots \partial_{\uc^{(N-1)}}\left(x_1^{N-1}x_2^{N-2}\cdots x_{N-1}\right).
\end{eqnarray*}
Consecutive applications of Lemma \ref{lem5-1} give
\[
\left[Y_{\uc^{(N-1)}\cdots \uc^{(n)}}^{(N)} \to \Fl_N\right] = \left(  \prod_{i=1}^{n-1} x_i^{n-i}\right)  \left(1+\gamma \sum_{k=n+1}^{N-1} e_2(x_1,\dots, x_k)\right).
\] 
By Lemma \ref{lem5-2} we can rewrite this expression (modulo $\bbS_N$) as:
\begin{equation}\label{eq5-3}
\left[Y_{\uc^{(N-1)}\cdots \uc^{(n)}}^{(N)} \to \Fl_N\right] = \left(  \prod_{i=1}^{n-1} x_i^{n-i}\right)  \left(1-\gamma  \sum_{i=n+1}^{N-1} (i-n)x_{i+1}e_1(x_1,\dots, x_i)\right).
\end{equation}
The right hand side of (\ref{eq5-2}) is well-defined as an element of $\bbI[x]_{bd}$ and it projects to (\ref{eq5-3}) for all $N\geq n$. This completes the proof.
\end{proof}
In view of Corollary \ref{cor3-7}, all stable Bott--Samelson classes can be computed from (\ref{eq5-2}) by applying divided difference operators. More precisely, pick $w \in W_n$ and $\uv\in \uW_n$ such that $v = w_0^{(n)}w$. Then
\[
\BS_w^{\uv} = \partial_{\uv} \BS_{w_0^{(n)}}^{1}.
\]
Moreover, since the second factor of (\ref{eq5-2}) is symmetric in $x_1,\dots, x_n$, one simply has to identify $\partial_{\uv} \left( x_1^{n-1}x_2^{n-2} \cdots x_{n-1}\right)$. That is, 
\[
\BS_w^{\uv}  = \calB_n(x) \partial_{\uv}\left(x_1^{n-1}x_2^{n-2}\cdots x_{n-1}\right),
\]
where we denote
\[
\calB_n(x):=1-\gamma \sum_{i=n+1}^{\infty} (i-n)x_{i+1}e_1(x_1,\dots, x_i).
\]
Based on this, we will now obtain explicit closed formulas for the power series representing the stable Bott--Samelson classes associated to dominant permutations. 

\begin{defn}\label{defnDOM1}
For a permutation $w\in S_n$, consider a $n\times n$ grid with dots in the boxes $(i,w(i))$. The diagram of $w$ is the set of boxes that remain after deleting boxes weakly east and south of each dot. A permutation $w$ is called {\it dominant} if its diagram is located at the NW corner of the grid, and coincides with a Young diagram of a partition $\lambda=(\lambda_1,\dots,\lambda_r)$ with $\lambda_i\leq n-i$. For a given such partition $\lambda$, there is a unique dominant permutation $w_{\lambda} \in S_n$. For example, the longest element $w_0^{(n)}$ is dominant and its associated partition is $\rho:=(n-1,n-2,\dots,2,1)$.

Let $T$ be the standard tableau of $\rho$, {\it i.e.}, the fillings of the boxes of the $i$-th row of $T$ are all $i$. One places $\lambda$ at the NW corner of $\rho$ with its boxes shaded. We order the anti-diagonals starting from the  the inner ones to the outer ones, {\it i.e.,} the $i$-th anti-diagonal consists of boxes at $(a,b)$ with $a+b=n+2-i$. Let $m$ be the biggest number such that the $m$-th anti-diagonal contains unshaded boxes. Let $\uv^{(i)}, 1\leq i\leq m$ be the reduced word obtained by reading the numbers in the $i$-th anti-diagonal. Then $\uv:=\uv^{(1)}\cdots \uv^{(m)}$ is a reduced word of $v=w_0^{(n)} w_{\lambda}$. For each $i$, let $\bfx^{(i)}_k \ (k=1,\dots, a_i)$ be the orbits of $v^{(i)}$ in $\{x_1,\dots, x_n\}$ with cardinality greater than $1$.
\end{defn}
\begin{thm}
Let $w_{\lambda} \in S_n$ be the dominant permutation associated to the partition $\lambda=(\lambda_1,\dots,\lambda_r)$. Let $\uv:=\uv^{(1)}\cdots \uv^{(m)}$ be the reduced word of $v=w_0^{(n)}w_{\lambda}$ constructed in Definition \ref{defnDOM1}. We have
\[
\BS_{w_{\lambda}}^{\uv} = x_1^{\lambda_1}\cdots x_r^{\lambda_r} \left(1+\gamma \left(
\sum_{i=1}^m \sum_{k=1}^{a_i} e_2(\bfx^{(i)}_k)\right)\right)\calB_n(x).
\]
\end{thm}
\begin{proof}
We prove the formula by induction on $m$. If $m=1$, then by Lemma \ref{lem5-1} we have
\[
\partial_{\uv^{(1)}}\left(x_1^{n-1}x_2^{n-2}\cdots x_{n-1}\right) = x_1^{\lambda_1}\cdots x_r^{\lambda_r} \left(1+\gamma \left(\sum_{k=1}^{a_1} e_2(\bfx^{(1)}_k)\right)\right).
\]
Now, let $m>1$. By the induction hypothesis, we have
\[
\partial_{\uv^{(m)}}\cdots \partial_{\uv^{(1)}}\left(x_1^{n-1}x_2^{n-2}\cdots x_{n-1}\right) 
=\partial_{\uv^{(m)}}\left(x_1^{\lambda_1'}\cdots x_r^{\lambda_r'} \left(1+\gamma \left(
\sum_{i=1}^{m-1} \sum_{k=1}^{a_i} e_2(\bfx^{(i)}_k)\right)\right)\right),
\]
where $\lambda'=\lambda \cup (n-m,n-m-1,\dots, 2,1,)$. Since the unshaded boxes form a skew shape $\rho/\lambda$, it follows that $\uv^{(m)}$ stabilizes the second factor, allowing it pass through $\partial_{\uv^{(m)}}$:
\[
\partial_{\uv^{(m)}}\cdots \partial_{\uv^{(1)}}\left(x_1^{n-1}x_2^{n-2}\cdots x_{n-1}\right) 
=\left(1+\gamma \left(
\sum_{i=1}^{m-1} \sum_{k=1}^{a_i} e_2(\bfx^{(i)}_k)\right)\right)\cdot\partial_{\uv^{(m)}}\left(x_1^{\lambda_1'}\cdots x_r^{\lambda_r'} \right).
\]
Now the desired formula follows again from Lemma \ref{lem5-1}.
\end{proof}
\begin{exm}
Consider $w_{\lambda}=(53124) \in S_5$ where $\lambda=(4,2)$. 
\setlength{\unitlength}{0.5mm}
\begin{center}
\begin{picture}(40,40)
\put(00,30){\textcolor{grey}{\rule{40\unitlength}{10\unitlength}}}
\put(00,20){\textcolor{grey}{\rule{20\unitlength}{10\unitlength}}}
\put(00,40){\line(1,0){40}}
\put(00,30){\line(1,0){40}}
\put(00,20){\line(1,0){30}}
\put(00,10){\line(1,0){20}}
\put(00,00){\line(1,0){10}}
\put(00,40){\line(0,-1){40}}
\put(10,40){\line(0,-1){40}}
\put(20,40){\line(0,-1){30}}
\put(30,40){\line(0,-1){20}}
\put(40,40){\line(0,-1){10}}
\put(03,32){1}\put(13,32){1}\put(23,32){1}\put(33,32){1}
\put(03,22){2}\put(13,22){2}\put(23,22){2}
\put(03,12){3}\put(13,12){3}
\put(03,02){4}
\end{picture}
\end{center}
The reduced word $\uv$ of $v=w_0^{(5)}w_{\lambda}$ is $\uv=(s_2s_3s_4)(s_3)$. We have 
\[
(\partial_3)(\partial_4\partial_3\partial_2)\left(x_1^4x_2^3x_3^2x_4\right) =(1+\gamma e_2([2,5])\cdot \partial_3\left(x_1^4x_2^2x_3\right)=x_1^4x_2^2\left(1+\gamma (e_2^{[2,5]}+e_2^{[3,4]})\right).
\]
\end{exm}
\begin{exm}
Consider $w_{\lambda}=(45123) \in S_5$ where $\lambda=(3,3)$.
\setlength{\unitlength}{0.5mm}
\begin{center}
\begin{picture}(40,40)
\put(00,30){\textcolor{grey}{\rule{30\unitlength}{10\unitlength}}}
\put(00,20){\textcolor{grey}{\rule{30\unitlength}{10\unitlength}}}
\put(00,40){\line(1,0){40}}
\put(00,30){\line(1,0){40}}
\put(00,20){\line(1,0){30}}
\put(00,10){\line(1,0){20}}
\put(00,00){\line(1,0){10}}
\put(00,40){\line(0,-1){40}}
\put(10,40){\line(0,-1){40}}
\put(20,40){\line(0,-1){30}}
\put(30,40){\line(0,-1){20}}
\put(40,40){\line(0,-1){10}}
\put(03,32){1}\put(13,32){1}\put(23,32){1}\put(33,32){1}
\put(03,22){2}\put(13,22){2}\put(23,22){2}
\put(03,12){3}\put(13,12){3}
\put(03,02){4}
\end{picture}
\end{center}
The reduced word of $v=w_0^{(5)}w_{\lambda}$ is $\uv=(s_1s_3s_4)(s_3)$. We have 
\[
(\partial_3)(\partial_4\partial_3\partial_1)\left(x_1^4x_2^3x_3^2x_4\right) 
= \left(1+\gamma (e_2^{[1,2]} + e_2^{[3,5]})\right)\partial_3\left(x_1^3x_2^3x_3\right)
= x_1^3x_2^3\left(1+\gamma(e_2^{[1,2]}+e_2^{[3,5]}+e_2^{[3,4]})\right).
\]
\end{exm}
\begin{exm}
Consider $w_{\lambda}=(563412) \in S_6$ where $\lambda=(4,4,2,2)$.
\setlength{\unitlength}{0.5mm}
\begin{center}
\begin{picture}(50,50)
\put(00,40){\textcolor{grey}{\rule{40\unitlength}{10\unitlength}}}
\put(00,30){\textcolor{grey}{\rule{40\unitlength}{10\unitlength}}}
\put(00,20){\textcolor{grey}{\rule{20\unitlength}{10\unitlength}}}
\put(00,10){\textcolor{grey}{\rule{20\unitlength}{10\unitlength}}}
\put(00,50){\line(1,0){50}}
\put(00,40){\line(1,0){50}}
\put(00,30){\line(1,0){40}}
\put(00,20){\line(1,0){30}}
\put(00,10){\line(1,0){20}}
\put(00,00){\line(1,0){10}}
\put(00,50){\line(0,-1){50}}
\put(10,50){\line(0,-1){50}}
\put(20,50){\line(0,-1){40}}
\put(30,50){\line(0,-1){30}}
\put(40,50){\line(0,-1){20}}
\put(50,50){\line(0,-1){10}}
\put(03,42){1}\put(13,42){1}\put(23,42){1}\put(33,42){1}\put(43,42){1}
\put(03,32){2}\put(13,32){2}\put(23,32){2}\put(33,32){2}
\put(03,22){3}\put(13,22){3}\put(23,22){3}
\put(03,12){4}\put(13,12){4}
\put(03,02){5}
\end{picture}
\end{center}
The reduced word of $v=w_0^{(6)}w_{\lambda}$ is $\uv=s_1s_3s_5$. We have 
\[
\partial_5\partial_3\partial_1\left(x_1^5x_2^4x_3^3x_4^2x_5\right) = (x_1x_2)^4(x_3x_4)^2\left(1+\gamma(e_2^{[1,2]}+e_2^{[3,4]}+e_2^{[5,6]})\right). 
\]
\end{exm}

\section{Restriction of Bott--Samelson classes}\label{sec:rest}
In this section we generalise the restriction formula in Theorem \ref{thm:restriction} of the previous section. Namely, we will prove a formula for the product of the cobordism class of any Bott--Samelson resolution with the class $[\Fl_{n-1} \to \Fl_n]$. In order to simplify the proof we will use another equivalent definition of Bott--Samelson resolutions.
\subsection{Bott--Samelson resolution revisited}
In this section, we provide another construction of the Bott--Samelson variety $X_\wu$ associated to a word $\wu$ by viewing it as a configuration space. This description will be better suited for our purposes.
\begin{defn}\label{def:BS2}
  Let $\wu = s_{i_1} \cdots s_{i_r}$ be a word in $\Wu_{n+1}$.
  \begin{enumerate}
\item For $a \in [0,n+1]$, define $\lo_\wu(a)$, the last occurence of $a$ in $\wu$, by
  $$\lo_\wu(a) = \sup\{k \in [1,r] \ | \ i_k = a \}.$$
Note that if the above set is empty, then $\lo_\wu(a) = - \infty$.

\item  If $\lo_\wu(a) = - \infty$, then we set $V_{\lo_\wu(a)} = \scal{e_1,\cdots,e_a}$.

\item For $k \in [1,r]$, define $\wu[k] = s_{i_1} \cdots s_{i_k}$.
  
\item For $k \in [1,r]$, the left and the right predecessors of $k$ in $\wu$, denoted $\lp_\wu$ and $\rp_\wu$, are defined as: 
$$\begin{array}{l}
      \lp_\wu(k) := \lo_{\wu[k]}(i_k - 1), \\
      \rp_\wu(k) := \lo_{\wu[k]}(i_k + 1).\\
  \end{array}$$

\item We set
  $$\begin{array}{l}
    V_{\lp_\wu(k)} = \scal{e_1,\cdots,e_{i_k - 1}} \textrm{ if } \lp_\wu(k) = - \infty \textrm{ and }\\
    V_{\rp_\wu(k)} = \scal{e_1,\cdots,e_{i_k + 1}} \textrm{ if } \rp_\wu(k) = - \infty. \\
  \end{array}$$
  \end{enumerate}
\end{defn}
\begin{defn}
  Given a word $\wu = s_{i_1} \cdots s_{i_r}$, define the Bott--Samelson variety $X_\wu$ as follows:
  $$X_\wu = \Big\{(V_k)_{k \in [1,r]} \ | \ \dim V_k = i_k \textrm{ and } V_{\lp_\wu(k)} \subset V_k \subset V_{\rp_\wu(k)} \Big\}.$$
  Define a morphism $\pi_\wu : X_\wu \to \Fl_{n+1}$ by $\pi_\wu((V_k)_{k \in [1,r]}) = (V_{\lo_\wu(a)})_{a \in [1,n]}$. If $\wu$ is reduced, then the map $\pi_\wu$ is a proper birational morphism from $X_\wu$ onto the Schubert variety $X_w$. In this reduced case, we often call $X_\wu$ together with the map $\pi_\wu$ a Bott--Samelson resolution.
\end{defn}
\begin{rem}
There is a natural isomorphism between our two construction of the Bott--Samelson variety $X_\wu$ and $Y_\wu$. It is given by $(V_k)_{k \in [1,r]} \mapsto (U_\bullet^{(k)})_{k \in [1,r]}$ with $U_a^{(k)} = V_{\lo_{\wu[k]}(a)}$ for all $k \in [1,r]$ and $a \in [1,n]$.
\end{rem}
Recall the following well known fact on Bott--Samelson varieties.
\begin{lem}\label{lem4-4}
The Bott--Samelson variety does not depend on the choice of a word modulo commuting relations. More precisely, if $\wu = \vu$ modulo commuting relations, then there is an isomorphism $f_{\wu,\vu} : X_\wu \to X_\vu$ such that the following diagram is commutative:
    $$\xymatrix{X_\wu \ar[r]^-{f_{\wu,\vu}} \ar[d]_-{\pi_\wu} & X_\vu \ar[d]_-{\pi_\vu} \\
    X_w \ar[r]^-{\rm Id} & X_v.}$$
\end{lem}
\begin{proof}
  It is enough to prove this result in the case in which $\wu$ and $\vu$ are obtained from each other by a unique commuting relation. The result then follows by induction on the number of commuting relations. Assume therefore that $\wu = \uu_1 s_a s_b \uu_2$ and $\vu = \uu_1 s_b s_a \uu_2$ with $| a - b | \geq 2$. Let $r_i = \ell(\uu_i)$ for $i \in [1,2]$ and set $r = \ell(\wu) = \ell(\vu) = r_1 + r_2 + 2$. Define the map $f_{\wu,\vu} : X_\wu \to X_\vu$ by $f_{\wu,\vu}((V_k)_{k \in [1,r]}) = (W_k)_{k \in [1,r]}$ and $f_{\vu,\wu}((W_k)_{k \in [1,r]}) = (V_k)_{k \in [1,r]}$ with $W_k = V_k$ for $k \not\in \{r_1 + 1,r_1 + 2\}$ and $W_{r_1 + \epsilon} = V_{r_1 + 3 - \epsilon}$ for $\epsilon \in \{1,2\}$. These maps are inverses of each other and we only need to check that they indeed map $X_\wu$ to $X_\vu$ and $X_\vu$ to $X_\wu$ respectively. By symmetry, we only need to check this for $f_{\wu,\vu}$.

  We prove that given $(V_k)_{k \in [1,r]} \in X_\wu$, the condition $W_{\lp_\vu(k)} \subset W_k$ is satisfied for all $k$. The other inclusion $W_k \subset W_{\rp_\vu(k)}$ is obtained by similar arguments. First note that we have the following relations:
  $$\begin{array}{ll}
    \lp_\vu(k) = \lp_\wu(k) & \textrm{ for } k , \lp_\wu(k) \not\in \{r_1 + 1 , r_1 + 2\} \\
    \lp_\vu(k) = r_1 + 3 - \epsilon & \textrm{ for } \lp_\wu(k) = r_1 + \epsilon \textrm{ and } \epsilon \in \{1,2\} \\
    \lp_\vu(r_1 + \epsilon) = \lp_\wu(r_1 + 3 - \epsilon) & \textrm{ for } \epsilon \in \{1,2\} \\
  \end{array}$$
  For $k$ such that $k , \lp_\wu(k) \not\in \{r_1 + 1 , r_1 + 2\}$, we have $W_{\lp_\vu(k)} = W_{\lp_\wu(k)} = V_{\lp_\wu(k)} \subset V_k = W_k$. For $\lp_\wu(k) = r_1 + \epsilon$ with $\epsilon \in \{1,2\}$, note that $k \not\in \{r_1+1,r_1+2\}$ thus we have $W_{\lp_\vu(k)} = W_{r_1 + 3 - \epsilon} = V_{r_1 + \epsilon} =V_{\lp_\wu(k)} \subset V_k = W_k$. Finally, for $k = r_1 + \epsilon$ with $\epsilon \in \{1,2\}$, note that $\lp_\wu(r_1 + 3 - \epsilon) \not\in \{r_1+1,r_1+2\}$ thus we have $W_{\lp_\vu(k)} = W_{\lp_\vu(r_1 + \epsilon)} = W_{\lp_\wu(r_1 + 3 - \epsilon)} = V_{\lp_\wu(r_1 + 3 - \epsilon)} \subset V_{r_1 + 3 - \epsilon} = W_{r_1 + \epsilon}  = W_k$.
Furthermore we have:
  $$\begin{array}{ll}
    \lo_\wu(a) = \lo_\vu(a) & \textrm{ if } \lo_\wu(a) \not\in \{r_1 + 1 , r_1 + 2\}  \\
    \lo_\wu(a) = \lo_\vu(a) + 3 - \epsilon & \textrm{ if } \lo_\wu(a) = r_1 + \epsilon \textrm{ for } \epsilon \in \{1,2\}, \\
\end{array}$$
so we easily see that we have $\pi_\vu \circ f_{\wu,\vu} = \pi_\wu$ and $\pi_\wu \circ f_{\vu,\wu} = \pi_\vu$.
\end{proof}
\subsection{Fiber product with a subflag}
We now prove a fiber product formula for Bott--Samelson resolutions.

Define
\begin{equation}
\bfF_n = \Big\{ U_\bullet \in \Fl_{n+1} \ | \ U_n = \scal{e_2,\cdots,e_{n+1}}=F^{(n+1)}_n \Big\}.
\end{equation}
We can easily see that $\bfF_n$ coincides with the opposite Schubert variety $X^c$ in $\Fl_{n+1}$ where  $c:=c^{(n)}=s_1\dots s_n$ is the Coxeter element. Therefore, from a well-known fact, we have that for $w\in W_{n+1}$
\begin{equation}
\mbox{$X_w \cap \bfF_n \neq \emptyset$ if and only if $c \leq w$}.
\end{equation}
\begin{defn}
For $\uu \in \Wu_{n+1}^1$ with $\uu = s_{i_1} \cdots s_{i_r}$, define $c^{-1}(\uu) \in \Wu_{n+1}^n$ by
\[
c^{-1}(\uu) = s_{c^{-1}(i_1)} \cdots s_{c^{-1}(i_r)} = s_{i_1 - 1} \cdots s_{i_r - 1},
\]
where we observe that for each $k \in [2,n+1]$ one has $c^{-1}(k) = k-1 \in [1,n]$.
\end{defn}
\begin{defn}
  Denote by $c$ the isomorphism $c : \kk^{n+1} \to \kk^{n+1}$ defined by $c(e_i) = e_{c(i)}$ for all $i \in [1,n+1]$.
  \begin{enumerate}
  \item The map $c$ induces an isomorphism $c : \Fl_{n} \to \bfF_n \subset \Fl_{n+1}$.    \item For $\wu \in W_{n+1}^n$ with $\ell(\wu) = r$, define
    $$c(X_\wu) = \Big\{ (c(V_k))_{k \in [1,r]} \ | \ (V_k)_{k \in [1,r]} \in X_\wu\Big\}$$
    and $c(\pi_\wu) : c(X_\wu) \to c(X_w) \subset c(\Fl_n) = \bfF_n$, so that the following diagram is commutative:
    $$\xymatrix{X_\wu \ar[r]^-c \ar[d]_-{\pi_\wu} & c(X_\wu) \ar[d]^-{c(\pi_\wu)} \\
    X_w \ar[r]^-c & c(X_w).}$$
  \end{enumerate}
\end{defn}
\begin{lem}
  For $\uu \in \Wu_{n+1}^1$ and $\vu \in \Wu_{n+1}^n$, define $\wu = \uu\, \uc\, \vu$ and $\wu' = c^{-1}(\uu)\vu$. Let $r_1 = \ell(\uu)$ and $r_2 = \ell(\vu)$. For $\a \in [1,n-1]$, we have the following equalities:
  \begin{enumerate}
    \item If $k \leq r_1$, then $\lo_{\wu'[k]}(\a) = \lo_{\wu[k]}(\a+1)$;
  \item If $k \geq r_1 + 1$ and $\lo_{\vu[k-r_1]}(\a) \neq - \infty$, then $\lo_{\wu'[k]}(\a) = \lo_{\wu[k+n]}(\a) - n = \lo_{\vu[k-r_1]}(\a) + r_1 \geq r_1 + 1$;
    \item If $k \geq r_1 + 1$ and $\lo_{\vu[k-r_1]}(\a) = - \infty$, then $\lo_{\wu'[k]}(\a) = \lo_\uu(\a + 1) = \rp_{\wu[k+n]}(r_1+\a) \leq r_1$ and $\lo_{\wu[k+n]}(\a) = r_1 + \a$.
  \end{enumerate}
\end{lem}
\begin{proof}
If $k \leq r_1$, then $\wu'[k] = c^{-1}(\wu[k])$ and the result follows. Assume now that $k \geq r_1 + 1$. If $\lo_{\vu[k-r_1]}(\a) \neq - \infty$, then the last occurence of $\a$ in $\wu'[k]$ is obtained at a letter of $\vu[k-r_1]$ so that $\lo_{wu'[k]}(\a) = \lo_{\vu[k-r_1]}(\a) + r_1 = \lo_{\wu[k+n]}(\a) - n \geq r_1$. If $\lo_{\vu[k-r_1]}(\a) = - \infty$, the last occurence of $\a$ in $\wu'[k]$ is the last occurence of $\a$ in $c^{-1}(\uu)$ and we get the equalities $\lo_{\wu'[k]}(\a) = \lo_{c^{-1}(\uu)}(\a) = \lo_\uu(\a+1) \leq r_1$. On the other hand, the last occurence of $\a$ in $\wu[k+n]$ is obtained as the $\a$-th letter in $\uc$ thus $\lo_{\wu[k+n]}(\a) = r_1 + \a$. This also explains the equality $\lo_\uu(\a+1) = \rp_{\wu[k+n]}(r_1+\a)$.
\end{proof}
\begin{cor}
  For $\uu \in \Wu_{n+1}^1$ and $\vu \in \Wu_{n+1}^n$, define $\wu = \uu \uc \vu$ and $\wu' = c^{-1}(\uu)\vu$. Let $r_1 = \ell(\uu)$ and $r_2 = \ell(\vu)$. We have the following alternatives:
  \begin{enumerate}
  \item If $\lo_\vu(\a) \neq - \infty$, then $\lo_{\wu'}(\a) = \lo_\wu(\a) - n = \lo_\vu(\a) + r_1 \geq r_1 + 1$;
    \item If $\lo_\vu(\a) = - \infty$, then $\lo_{\wu'}(\a) = \lo_\uu(\a + 1) = \rp_\wu(r_1+\a)$ and $\lo_\wu(\a) = r_1 + \a$.
  \end{enumerate}
\end{cor}
\begin{proof}
Apply the previous lemma with $k = r_1 + r_2$.
\end{proof}
\begin{cor}
For $\uu \in \Wu_{n+1}^1$ and $\vu \in \Wu_{n+1}^n$, define $\wu = \uu \uc \vu$ and $\wu' = c^{-1}(\uu)\vu$. Let $r_1 = \ell(\uu)$ and $r_2 = \ell(\vu)$.
We set $\lp_{\underline{x}}(-\infty) = -\infty$ and $\rp_{\underline{x}}(-\infty) = -\infty$ for any word $\underline{x}$.
\begin{enumerate}[A.]
\item We have the following formulas for $\lp_\wu$ and $\lp_{\wu'}$:
    \begin{enumerate}[$1.$]
    \item If $a \in [1,r_1]$, then $\lp_{\wu'}(a) = \lp_\wu(a) \leq r_1$.
    \item If $a > r_1$ and $\lp_\vu(a-r_1) \neq \infty$, then $\lp_{\wu'}(a) = \lp_\wu(a+n) - n > r_1$.
    \item If $a > r_1$ and $\lp_\vu(a - r_1) = - \infty$, then $\lp_\wu(a+n) \in [r_1 + 1,r_1 + n] \cup \{-\infty\}$ and $\lp_{\wu'}(a) = \rp_\wu(\lp_\wu(a+n)) \leq r_1$.
    \end{enumerate}

\item We have the following formulas for $\rp_\wu$ and $\rp_{\wu'}$:
    \begin{enumerate}[$1.$]
    \item If $a \in [1,r_1]$, then $\rp_{\wu'}(a) = \rp_\wu(a) \leq r_1$.
    \item If $a > r_1$ and $\rp_\vu(a-r_1) \neq \infty$, then $\rp_{\wu'}(a) = \rp_\wu(a+n) - n > r_1$.
      \item If $a > r_1$ and $\rp_\vu(a - r_1) = - \infty$, then $\rp_{\wu'}(a) = \rp_\wu(\rp_\wu(a+n)) \leq r_1$ and  $\rp_\wu(a+n) \in [r_1 + 1,r_1 + n] \cup \{-\infty\}$.
    \end{enumerate}

\item If $k \in [r_1 + 1, r_1 + n]$, then $\lp_\wu(k) \in [r_1 + 1,r_1 + n] \cup\{-\infty\}$ Furthermore, for $k \in [r_1 + 1, r_1 + n]$, we have the following formulas for $\lp_\wu$ and $\rp_{\wu}$:
\begin{enumerate}[$1.$]
\item If $\rp_\wu(k) > \rp_\wu(\lp_\wu(k))$, then  $\rp_\wu(\lp_\wu(k)) = \lp_\wu(\rp_{\wu}(k))$.
  \item If $\rp_\wu(k) < \rp_\wu(\lp_\wu(k))$, then $\rp_\wu(\rp_\wu(\lp_\wu(k))) = \rp_{\wu}(k)$.
\end{enumerate}
\end{enumerate}
\end{cor}
\begin{proof}
Write $\uu = s_{i_1} \cdots s_{i_{r_1}}$ and $\vu = s_{i_{r_1 + n + 1}} \cdots s_{i_{r_1 + n + r_2}}$ so that $w = s_{i_1} \cdots s_{i_r}$ with $r = r_1 + r_2 + n$ and $s_{i_{r_1 + k}} = s_k$ for $k \in [1,n]$. We have $\wu' = s_{j_1} \cdots s_{j_{r_1 + r+2}}$ with
  $$j_k = \left\{\begin{array}{ll}
  i_k - 1 & \textrm{ for } k \in [1,r_1]\\
  i_{k+n} & \textrm{ for } k \in [r_1 + 1,r_1 + r_2].\\
\end{array}\right.$$

A.1. We have $\lp_{\wu'}(a) = \lo_{\wu'[a]}(j_a - 1) = \lo_{\wu[a]}((j_a - 1)+1) = \lo_{\wu[a]}(j_a) = \lo_{\wu[a]}(i_a - 1) = \lp_\wu(a)$. By definition $\lp_{\wu'}(a) < a \leq r_1$.

A.2. We have $\lp_{\wu'}(a) = \lo_{\wu'[a]}(j_a - 1) = \lo_{\wu[a+n]}(i_{a+n}-1) - n= \lp_\wu(a+n)-n \geq r_1 + 1$.

A.3. We have $\lp_{\wu'}(a) = \lo_{\wu'[a]}(j_a - 1) = \rp_{\wu[a+n]}(r_1 + j_a - 1) \leq r_1$ and $r_1 + j_a - 1 = \lo_{\wu[a+n]}(j_a-1)$. We get $\lp_{\wu'}(a) = \rp_{\wu[a+n]}(\lo_{\wu[a+n]}(j_a-1)) = \rp_{\wu[a+n]}(\lo_{\wu[a+n]}(i_{a+n}-1)) = \rp_{\wu[a+n]}(\lp_{\wu}(a+n))$.

For B.1. and B.2. use the proof of A.1. and A.2 with $\rp$ in place of $\lp$.

B.3. We have $\rp_{\wu'}(a) = \lo_{\wu'[a]}(j_a + 1) = \rp_{\wu[a+n]}(r_1 + j_a + 1) \leq r_1$ and $r_1 + j_a + 1 = \lo_{\wu[a+n]}(j_a+1)$. We get $\lp_{\wu'}(a) = \rp_{\wu[a+n]}(\lo_{\wu[a+n]}(j_a+1)) = \rp_{\wu[a+n]}(\lo_{\wu[a+n]}(i_{a+n}+1)) = \rp_{\wu[a+n]}(\rp_{\wu}(a+n))$.

C. If $k \in [r_1+1,r_1+n]$, then the $k$-th letter of $\wu$ is the $(k - r_1)$-th letter of $\uc$ and we have $\lp_\wu(k) = \lp_\uc(k-r_1) = k - r_1 - 1$ for $k > r_1 + 1$ and $\lp_{\wu}(r_1 + 1)  = -\infty$. Note that, if any of the two quatities $\rp_\wu(k)$ or  $\rp_\wu(\lp_\wu(k))$ is finite, we have $\rp_\wu(k) \neq \rp_\wu(\lp_\wu(k))$.

C.1. Assume $\rp_\wu(k) > \rp_\wu(\lp_\wu(k))$. If $k = r_1 + 1$, then $i_k = 1$ and $\lp_\wu(k) = -\infty$ thus $\rp_\wu(\lp_\wu(k)) = -\infty$. Furthermore, $i_{\rp_\wu(k)} = i_k + 1 = 2$ and since $\uu \in \Wu_{n+1}^1$ we have $\lp_\wu(\rp_\wu(k)) = \lo_\uu(2) = -\infty$. Assume now $k > r_1 + 1$, then $\lp_\wu(k) = k-1$, thus $\rp_\wu(\lp_\wu(k)) = \rp_\wu(k-1) = \lo_\uu(i_{k-1}+1)=\lo_\uu(i_k)$. We have $i_{\rp_\wu(k)} = i_k + 1$ thus $\lp_\wu(\rp_\wu(k)) = \lo_{\wu[\rp_\wu(k)]}(i_{\rp_\wu(k)} - 1) = \lo_{\wu[\rp_\wu(k)]}(i_k)$. Since $\rp_\wu(k) < k$, we have that $\wu[\rp_\wu(k)]$ is a subword of $\uu$ thus $\lo_{\wu[\rp_\wu(k)]}(i_k) \leq \lo_{\uu}(i_k) = \rp_\wu(\lp_\wu(k))$. On the other hand, since $\rp_\wu(k) > \rp_\wu(\lp_\wu(k))$, we have $\lo_{\wu[\rp_\wu(k)]}(i_k) \geq \lo_{\wu[\rp_\wu(\lp_\wu(k))]}(i_k) = \rp_\wu(\lp_\wu(k))$. The last equality holds since $i_{\rp_\wu(\lp_\wu(k))} = i_k$.

C.2. Assume $\rp_\wu(k) < \rp_\wu(\lp_\wu(k))$. This implies $\lp_\wu(k) \neq -\infty$ thus $k > r_1 + 1$. We have $\rp_\wu(k) = \lo_{\wu[k]}(i_k+1)$ and $\rp_\wu(\rp_\wu(\lp_\wu(k))) = \lo_{\wu[\rp_\wu(\lp_\wu(k))]}(i_k+1)$.
Since $\rp_\wu(\lp_\wu(k)) < k$, then $\wu[\rp_\wu(\lp_\wu(k))]$ is a subword of $\wu[k]$ thus $\lo_{\wu[\rp_\wu(\lp_\wu(k))]}(i_k+1) \leq \lo_{\wu[k]}(i_k+1)$. On the other hand, since $\rp_\wu(k) < \rp_\wu(\lp_\wu(k))$, we have $\lo_{\wu[\rp_\wu(\lp_\wu(k))]}(i_k+1) \geq \lo_{\rp_\wu[k]}(i_k+1)$. 
  \end{proof}
\begin{thm}\label{thm-fp}
Let $\wu$ be a reduced word and $w \in W$ the associated element.
\begin{enumerate}[$(1)$]
\item If $w \not \geq c$, then $X_\wu \times_{\Fl_{n+1}} \bfF_n$ is empty.
\item If $w\geq c$, there exist $\uu \in \Wu_{n+1}^1$ and $\vu \in \Wu_{n+1}^n$ such that $\wu = \uu\, \uc\, \vu$ modulo commuting relations and we have an isomorphism of $\Fl_{n+1}$-varieties $X_\wu \times_{\Fl_{n+1}} \bfF_n \simeq c(X_{c^{-1}(\uu)\vu})$.
\end{enumerate}
\end{thm}
\begin{proof}
(1) is clear since the condition implies $X^w \cap \bfF_n = \emptyset$. We prove (2).  By Proposition \ref{prop:w=ucv}, we can write $\wu = \uu\, \uc\, \vu$. Let $\ell(w)=r$, and $\ell(u) = r_1$, $\ell(v) = r_2$ so that $r = r_1 + r_2 + n$. Since the obvious inclusion $\bfF_n \inc \Fl_{n+1}$ is a closed embedding, we can view $X_\wu \times_{\Fl_{n+1}} \bfF_n$ as the closed subvariety of $X_\wu$ given as follows:
$$X_\wu \times_{\Fl_{n+1}} \bfF_n = \Big\{ (V_k)_{k \in [1,r]} \in X_\wu \ | \ V_{r_1 + n} = \scal{e_2,\cdots,e_{n+1}} \Big\}.$$
Define the map $f : X_\wu \times_{\Fl_{n+1}} \bfF_n \to c(X_{c^{-1}(\uu)\vu})$ by $f((V_k)_{k \in [1,r]}) = (H_a)_{a \in [r_1 + r_2]}$ with
  $$H_a = \left\{\begin{array}{ll}
V_{a} \cap \scal{e_2,\cdots,e_{n+1}} & \textrm{ for } a \in [1, r_1] \\
V_{a + n} & \textrm{ for } a \in [r_1 + 1, r_1 + r_2]. \\
  \end{array}\right.$$
  Define the map $g : c(X_{c^{-1}(\uu)\vu}) \to X_\wu \times_{\Fl_{n+1}} \bfF_n$ by $g((H_a)_{a \in [r_1 + r_2]}) = (V_k)_{k \in [1,r]}$ with
\[
V_k = \left\{\begin{array}{ll}
  H_{k} + \scal{e_1} & \textrm{ for } k \in [1, r_1] \\
  V_{\rp_\wu(k)} \cap \scal{e_2,\cdots,e_{n+1}} & \textrm{ for } k \in [r_1 + 1, r_1 + n] \\
  H_{k - n} & \textrm{ for } k \in [r_1 + n + 1, r]. \\
  \end{array}\right.
  \]
  Set $\wu' = c^{-1}(\uu)\vu \in \Wu_{n+1}^n$. Write $\uu = s_{i_1} \cdots s_{i_{r_1}}$ and $\vu = s_{i_{r_1 + n + 1}} \cdots s_{i_{r_1 + n + r_2}}$ so that $w = s_{i_1} \cdots s_{i_r}$ with $s_{i_{r_1 + k}} = s_k$ for $k \in [1,n]$. We have $\wu' = s_{j_1} \cdots s_{j_{r_1 + r+2}}$ with
\[
j_k = \left\{\begin{array}{ll}
i_k - 1 & \textrm{ for } k \in [1,r_1]\\
i_{k+n} & \textrm{ for } k \in [r_1 + 1,r_1 + r_2].\\
\end{array}\right.
\]

    We first prove that these maps are well defined. We start with $f$. Note that if $(V_k)_{k \in [1,r]} \in X_\wu \times_{\Fl_{n+1}} \bfF_n$, then
  \[\begin{array}{ll}
    V_k \supset \scal{e_1} & \textrm{ for } k \in [1,r_1] \\
    V_k \subset \scal{e_2,\cdots,e_{n+1}} & \textrm{ for } k \in [r_1+1,r_1+r_2+n] \\
    V_k = V_{\rp_\wu(k)} \cap \scal{e_2,\cdots,e_{n+1}} & \textrm{ for } k \in [r_1 + 1,r_1+n]. \\
  \end{array}
  \]
  Indeed, since $\uu \in \Wu_1$, we have $\lp_\wu(k) = - \infty$ for any $k \in [1,r_1]$ with $i_k = 2$ implying our first claim. Furthermore, since $\vu \in \Wu_{n+1}$, by the same type of arguments we have the equality $V_{r_1+n} = \scal{e_2,\cdots,e_{n+1}}$ proving the second claim. In particular for $k \in [r_1 + 1,r_1+n]$, we have $V_k \subset V_{\rp_\wu(k)} \cap \scal{e_2,\cdots,e_{n+1}}$ and $\scal{e_1} \subset V_{\rp_\wu(k)}$. Since $\dim V_k = \dim V_{\rp_\wu(k)} - 1$ this proves the last equality.
  
  We check that $\dim H_a = j_a$. For $a \in [1,r_1]$, we have $\dim H_a = \dim(V_a \cap \scal{e_2,\cdots,e_{n+1}}) = \dim V_a - 1 = i_a - 1 = j_a$,
  where the second equality holds since $\uu \in \Wu_{n+1}^1$, therefore $\scal{e_1} \subset V_a$. For $a \in [r_1 + 1,r_1 + r_2]$, we have $\dim H_a = \dim V_{a+n} = i_{a+n} = j_a$.

  We now check the inclusions $H_{\lp_{\wu'}(a)} \subset H_a \subset H_{\rp_{\wu'}(a)}$. We start with the inclusions $H_{\lp_{\wu'}(a)} \subset H_a$. For $a \in [1,r_1]$, then $\lp_{\wu'}(a) = \lp_\wu(a) \leq r_1$ and we have $H_{\lp_{\wu'}(a)} = H_{\lp_{\wu}(a)} = V_{\lp_{\wu}(a)} \cap \scal{e_2,\cdots,e_{n+1}} \subset V_a \cap \scal{e_2,\cdots,e_{n+1}} = H_a$.
  For $a \in [r_1 + 1,r_1 + r_2]$ and $\lp_\vu(a - r_1) \neq - \infty$, then $\lp_{\wu'}(a) = \lp_\wu(a + n)-n \geq r_1+1$ and we have $H_{\lp_{\wu'}(a)} = H_{\lp_{\wu}(a+n)-n} = V_{\lp_{\wu}(a+n)} \subset V_{a+n}  = H_a$.
  For $a \in [r_1 + 1,r_1 + r_2]$ and $\lp_\vu(a) = - \infty$, then $\lp_{\wu'}(a) = \rp_\wu(\lp_\wu(a + n)) \leq r_1$ and $\lp_\wu(a+n) \in [r_1+1,r_1+n] \cup\{-\infty\}$. If $\lp_\wu(a+n) = - \infty$, then $\lp_{\wu'}(a) = -\infty$ and $H_{\lp_{\wu'}(a)} = 0$, so the inclusion holds. Otherwise, we have $H_{\lp_{\wu'}(a)} = H_{\rp(\lp_{\wu}(a))} = V_{\rp(\lp_{\wu}(a+n))} \cap \scal{e_2,\cdots,e_{n+1}} = V_{\lp_\wu(a+n)}  \subset V_a = H_a$.

  We prove the inclusions $H_k \subset H_{\rp_{\wu'}(a)}$. For $a \in [1,r_1]$, then $\rp_{\wu'}(a) = \rp_\wu(a) \leq r_1$ and we have $H_a = V_a \cap \scal{e_2,\cdots,e_{n+1}} \subset V_{\rp_{\wu}(a)} \cap \scal{e_2,\cdots,e_{n+1}} = H_{\rp_{\wu}(a)} = H_{\rp_{\wu'}(a)}$.
  For $a \in [r_1 + 1,r_1 + r_2]$ and $\rp_\vu(a - r_1) \neq - \infty$, then $\rp_{\wu'}(a) = \rp_\wu(a + n)-n \geq r_1+1$ and we have $H_a = V_{a+n} \subset V_{\rp_{\wu}(a+n)} = H_{\rp_{\wu}(a+n)-n} = H_{\rp_{\wu'}(a)}$.
  For $a \in [r_1 + 1,r_1 + r_2]$ and $\rp_\vu(a) = - \infty$, then $\rp_{\wu'}(a) = \rp_\wu(\rp_\wu(a + n)) \leq r_1$ and $\rp_\wu(a+n) \in [r_1+1,r_1+n]\cup\{-\infty\}$. If $\rp_\wu(a+n) = - \infty$, then $\rp_{\wu'}(a) = -\infty$ and $H_{\lp_{\wu'}(a)} = \scal{e_1,\cdots,e_{n+1}}$, so the inclusion holds. Otherwise, we have $H_a = V_{a+n} \subset V_{\rp_{\wu}(a+n)} = V_{\rp(\rp_{\wu}(a+n))} \cap \scal{e_2,\cdots,e_{n+1}} = H_{\rp_\wu(\rp_{\wu}(a+n))} = H_{\rp_{\wu'}(a)}$.

\medskip

We now prove that $g$ is well defined. Note that for all $a$, we have $H_a \subset \scal{e_2,\cdots,e_{n+1}}$. We first check the equalities $\dim V_k = i_k$ for all $k \in [1,r]$. For $k \in [1,r_1]$, we have $\dim V_k = \dim(H_k + \scal{e_1}) = \dim H_k + 1 = j_k + 1 = i_k$, where the second equality holds since $H_k \subset \scal{e_2,\cdots,e_{n+1}}$. For $k \in [r_1 + 1,r_1 + n]$, we have $\dim V_k = \dim(V_{\rp_\wu(k)} \cap \scal{e_2,\cdots,e_{n+1}}) = \dim V_{\rp_\wu(k)} - 1 = (i_k + 1) - 1 = i_k$, where the second equality holds since $\rp_\wu(k) \leq r_1$ for $k \in [r_1+1,r_1+n]$ thus $\scal{e_1} \subset V_{\rp_\wu(k)}$. For $k \in [r_1 + n + 1,r]$, we have $\dim V_k = \dim H_{k-n} = j_{k-n} = i_k$.

We now check, for $k \in [1,r]$, the inclusions $V_{\lp_{\wu}(k)} \subset V_k \subset V_{\rp_{\wu}(k)}$. We start with the inclusions $V_{\lp_{\wu}(k)} \subset V_k $. For $k \in [1,r_1]$, we have $\lp_\wu(k) = \lp_{\wu'}(k) \leq r_1$ thus $V_{\lp_{\wu}(k)} = V_{\lp_{\wu'}(k)} = H_{\lp_{\wu'}(k)} + \scal{e_1} \subset H_k + \scal{e_1} = V_k$. For $k \in [r_1 + 1, r_1 + n]$ and $\rp_\wu(k) > \rp_\wu(\lp_\wu(k))$, we have $\rp_\wu(\lp_\wu(k)) = \lp_\wu(\rp_{\wu}(k))$ and $\lp_\wu(k) \in [r_1 + 1,r_1 + n] \cup\{-\infty\}$. If $\lp_\wu(k) = -\infty$, then $V_{\lp_\wu(k)} = 0$ and the inclusion holds. Otherwise, we have $V_{\lp_\wu(k)} = V_{\rp_\wu(\lp_\wu(k))} \cap \scal{e_2,\cdots,e_{n+1}} =  V_{\lp_\wu(\rp_\wu(k))} \cap \scal{e_2,\cdots,e_{n+1}} \subset V_{\rp_\wu(k)} \cap \scal{e_2,\cdots,e_{n+1}} = V_k$.
For $k \in [r_1 + 1, r_1 + n]$ and $\rp_\wu(k) < \rp_\wu(\lp_\wu(k))$, we have $\rp_\wu(\rp_\wu(\lp_\wu(k))) = \rp_\wu(k) \leq r_1$. We have $V_{\lp_\wu(k)} = V_{\rp_\wu(\lp_\wu(k))} \cap \scal{e_2,\cdots,e_{n+1}} \subset V_{\rp_\wu(\rp_\wu(\lp_\wu(k)))} \cap \scal{e_2,\cdots,e_{n+1}} \subset V_{\rp_\wu(k)} \cap \scal{e_2,\cdots,e_{n+1}} = V_K$.
For $k \geq r_1 + n$ and $\lp_\vu(k - n - r_1) \neq - \infty$, we have $\lp_\wu(k) -n = \lp_{\wu'}(k - n) \geq r_1 + n + 1$ thus $V_{\lp_\wu(k)} = H_{\lp_\wu(k)-n} = H_{\lp_{\wu'}(k-n)} \subset H_{k-n} = V_k$. For $k \geq r_1 + n$ and $\lp_\vu(k - n - r_1) = - \infty$, we have $\rp_\wu(\lp_\wu(k)) = \lp_{\wu'}(k - n) \leq r_1 $ and $\lp_\wu(k) \in [r_1 + 1 , r_1 + n] \cup \{-\infty\}$. If $\lp_\wu(k) = - \infty$, then $V_{\lp_\wu(k)} = 0$ and the inclusion holds. Otherwise, we have $V_{\lp_\wu(k)} = V_{\rp_\wu(\lp_\wu(k))} \cap \scal{e_2,\cdots,e_{n+1}} = V_{\lp_{\wu'}(k-n)} \cap \scal{e_2,\cdots,e_{n+1}} = (H_{\lp_{\wu'}(k-n)} + \scal{e_1} ) \cap \scal{e_2,\cdots,e_{n+1}} = H_{\lp_{\wu'}(k-n)} \subset H_{k-n} = V_k$.

We finish with the inclusions $V_k \subset V_{\rp_{\wu}(k)}$. For $k \in [1,r_1]$, we have $\rp_\wu(k) = \rp_{\wu'}(k) \leq r_1$ thus $V_k = H_k + \scal{e_1} \subset H_{\rp_{\wu'}(k)} + \scal{e_1} = V_{\rp_{\wu'}(k)} = V_{\rp_{\wu}(k)}$.
For $k \in [r_1 + 1, r_1 + n]$, we have $V_k = V_{\rp_\wu(k)} \cap \scal{e_2,\cdots,e_{n+1}} \subset V_{\rp_\wu(k)}$.
For $k \geq r_1 + n$ and $\rp_\vu(k - n - r_1) \neq - \infty$, we have $\rp_\wu(k) - n = \rp_{\wu'}(k - n) \geq r_1 + n + 1$ thus $V_k = H_{k-n} \subset H_{\rp_{\wu'}(k-n)} = H_{\rp_\wu(k)-n} = V_{\rp_\wu(k)} $.
For $k \geq r_1 + n$ and $\rp_\vu(k - n - r_1) = - \infty$, we have $\rp_\wu(\rp_\wu(k)) = \rp_{\wu'}(k - n) \leq r_1 $ and $\rp_\wu(k) \in [r_1 + 1 , r_1 + n] \cup \{-\infty\}$. If $\rp_\wu(k) = - \infty$, then $V_{\lp_\wu(k)} = \scal{e_1,\cdots,e_{n+1}}$ and the inclusion holds. Otherwise, we have $V_k = H_{k - n} \subset H_{\rp_{\wu'}(k-n)} = V_{\rp_{\wu'}(k-n)} = V_{\rp_\wu(\rp_\wu(k))}$ But since $V_k = H_{k-n} \subset \scal{e_2,\cdots,e_{n+1}}$, we get $V_k \subset V_{\rp_\wu(\rp_{\wu}(k))} \cap \scal{e_2,\cdots,e_{n+1}} = V_{\rp_\wu(k)}$ where the last equality holds since $\rp_\wu(k) \in [r_1+1,r_1+n]$.

\medskip

Now we prove that $f$ and $g$ are inverse to each other and that $\pi_\wu = \pi_{\wu'} \circ c^{-1} \circ f$. We first prove that $g \circ f$ is the identity. Write $g \circ f((V_k)_{k \in [1,r]}) = (V'_k)_{k \in [1,r]}$ and $f((V_k)_{k \in [1,r]}) = (H_a)_{a \in [1,r_1+r_2]}$. For $k \in [1,r_1]$, we have $V'_k = H_k + \scal{e_1} = (V_k \cap \scal{e_2,\cdots,e_{n+1}}) + \scal{e_1}$.
But for such $k$, we have $\scal{e_1} \subset V_k$, this implies $(V_k \cap \scal{e_2,\cdots,e_{n+1}}) + \scal{e_1} = V_k$. For $k \in [r_1 + 1,r_1+n]$, we proceed by induction on $k$ and remark that $\rp_\wu(k) < k$. We have $V'_k = V'_{\rp_\wu(k)} \cap \scal{e_2,\cdots,e_{n+1}} = V_{\rp_\wu(k)} \cap \scal{e_2,\cdots,e_{n+1}} = V_k$. For $k \geq r_1 + n + 1$, we have $V'_k = H_{k-n} = V_k$. 

Next we prove that $f \circ g$ is the identity. Write $f \circ g((H_a)_{a \in [1,r_1+r_2]}) = (H'_a)_{a \in [1,r_1+r_2]}$ and $g((H_a)_{a \in [1,r_1+r_2]}) = (V_k)_{k \in [1,r]}$. For $a \in [1,r_1]$, we have $H'_a = V_a \cap \scal{e_2,\cdots,e_{n+1}} = (H_a + \scal{e_1}) \cap \scal{e_2,\cdots,e_{n+1}}$. But for such $a$, we have $H_a \subset \scal{e_2,\cdots,e_{n+1}}$ and this implies $(H_a + \scal{e_1}) \cap \scal{e_2,\cdots,e_{n+1}} = H_a$. For $a \in [r_1 + 1,r_1+r_2]$, we have $H'_a = V_{a+n} = H_a$.

Finally we check that $\pi_\wu = \pi_{\wu'} \circ c^{-1} \circ f$. Write $f((V_k)_{k \in [1,r]}) = (H_a)_{a \in [r_1+r_2]}$, $\pi_\wu((V_k)_{k \in [1,r]}) = (U_\a)_{\a \in [1,n]}$ and  $\pi_{\wu'}((H_a)_{a \in [1,r_1+r_2]}) = (U'_\a)_{\a \in [1,n]}$. We need to prove that $U_\a = U'_\a$ for all $\a \in [1,n]$. If $\lo_\vu(\a) \neq - \infty$, we have $\lo_{\wu'}(\a) = \lo_\wu(\a) - n \geq r_1 + 1$. We get $U'_\a = H_{\lo_{\wu'}(\a)} = V_{\lo_{\wu'}(\a)+n} = V_{\lo_\wu(\a)} = U_\a$. If $\lo_\vu(\a) = - \infty$, we have $\lo_{\wu'}(\a) = \lo_\uu(\a + 1) = \rp_\wu(r_1 + \a) \leq r_1$ and $\lo_\wu(\a) = r_1 + \a$. We get $U'_\a = H_{\lo_{\wu'}(\a)} = H_{\rp_\wu(r_1+\a)} = V_{\rp_\wu(r_1+\a)} \cap \scal{e_2;\cdots,e_{n+1}} = V_{r_1+\a} = V_{\lo_\wu(\a)} = U_\a$.
\end{proof}
\subsection{A product formula in cobordism}
As a consequence of Theorem \ref{thm-fp} we prove a product formula in the algebraic cobordism $\Omega^*(\Fl_{n+1})$.
\begin{cor}\label{coro-cobor}
Let $\wu$ be a reduced word and $w \in W$ the associated element.
\begin{enumerate}
\item If $w \not \geq c$, then $[X_\wu]\cdot [\bfF_n] = 0$ in $\Omega^*(\Fl_{n+1})$.
\item If $w \geq c$, 
there exist $\uu \in \Wu_{n+1}^1$ and $\vu \in \Wu_{n+1}^n$ such that $\wu = \uu\, \uc\, \vu$ modulo commuting relations and we have 
\[
[X_\wu] \cdot [\bfF_n]  = [X_{c^{-1}(\uu)\vu}].
\]
in $\Omega^*(\Fl_{n+1})$.
\end{enumerate}
\end{cor}
 \begin{proof}
   The product $[X_\wu] \cdot [\bfF_n]$ is given by pulling back the exterior product $X_{\wu} \times \bfF_n \to \Fl_{n+1} \times \Fl_{n+1}$ along the diagonal map $\Delta : \Fl_{n+1} \to \Fl_{n+1} \times \Fl_{n+1}$, see \cite[Remark 4.1.14]{LevineMorel}. We thus have $[X_\wu] \cdot [\bfF_n] = \Delta^*[X_{\wu} \times \bfF_n \to \Fl_{n+1} \times \Fl_{n+1}]$. Applying \cite[Corollary 6.5.5.1]{LevineMorel}, we get $\Delta^*[X_\wu \times \bfF_n \to \Fl_{n+1} \times \Fl_{n+1}] = [X_\wu \times_{\Fl_{n+1}} \bfF_n]$ in $\Omega^*(X)$. The result follows since $[c(X_{c^{-1}(\uu)\vu})] = [X_{c^{-1}(\uu)\vu}]$.
\end{proof}
As a special case, we recover the restriction formula in Theorem \ref{thm:restriction} as a product formula.
\begin{cor}
  Let $\wu = \uc\,\vu$ be a reduced word with $\vu \in \Wu_{n+1}^n$. Then we have the following formula in $\Omega^*(\Fl_{n+1})$:
 $$[X_\wu] \cdot [\bfF_n]  = [X_\vu].$$ 
\end{cor}
By reversing the order of the simple reflection $s_1,\cdots,s_n$ (or equivalently by conjugating with the element $w_0$) we also obtain the following results in $\Omega^*(\Fl_{n+1})$:
\begin{prop}
Let $\wu \in \Wu_{n+1}$ be a reduced word and $\uc' := s_n \cdots s_1$ a Coxeter element. Let $\bfF'_n = \{ U_\bullet \in \Fl_{n+1} \ | \ U_1 = \scal{e_{n+1}} \}$.
Then $w \not\geq c'$ is equivalent to $X_w \cap \bfF'_n = \emptyset$ and we have the following alternatives:
\begin{enumerate}
\item If $w \not\geq c'$, then $[X_\wu] \cdot [\bfF'_n]  =0$ in $\Omega^*(\Fl_{n+1})$.
\item If $w \geq c'$, then, modulo commuting relations, we have $\wu = \uu\, \uc' \vu$ with $\uu \in \Wu_{n+1}^1$ and $\vu \in \Wu_{n+1}^n$. Furthermore, we have
      $$[X_\wu] \cdot [\bfF'_n]  = [X_{{c'}^{-1}(\uu)\vu}]$$
      in $\Omega^*(\Fl_{n+1})$.
\end{enumerate}
In particular, if $\wu = \uc' \vu$ is reduced with $\vu \in W_{n+1}^1$, then we have the following formula in $\Omega^*(\Fl_{n+1})$:
\[
[X_\wu] \cdot [\bfF'_n]  = [X_\vu].
\] 
\end{prop}
\bibliography{references}{}
\bibliographystyle{acm}

\end{document}